\documentclass{article}

\usepackage[]{geometry}
\usepackage{amsthm, amsmath, amsfonts, mathrsfs, amssymb, bbm, mathtools}
\usepackage{xcolor}
\usepackage[colorlinks=true, linkcolor=blue, filecolor=magenta, urlcolor=cyan, citecolor=brown]{hyperref}
\usepackage{graphicx, subcaption}
\usepackage[shortlabels]{enumitem}
\usepackage{esvect}
\usepackage{resmes}
\usepackage{cleveref}
\usepackage{bbm}
\usepackage{authblk}
\makeatletter
\newcommand*\rel@kern[1]{\kern#1\dimexpr\macc@kerna}
\newcommand*\widebar[1]{%
	\begingroup
	\def\mathaccent##1##2{%
		\rel@kern{0.8}%
		\overline{\rel@kern{-0.8}\macc@nucleus\rel@kern{0.2}}%
		\rel@kern{-0.2}%
	}%
	\macc@depth\@ne
	\let\math@bgroup\@empty \let\math@egroup\macc@set@skewchar
	\mathsurround\z@ \frozen@everymath{\mathgroup\macc@group\relax}%
	\macc@set@skewchar\relax
	\let\mathaccentV\macc@nested@a
	\macc@nested@a\relax111{#1}%
	\endgroup
}
\makeatother

\newtheorem{thm}{Theorem}[section]
\newtheorem*{thm*}{Theorem}
\newtheorem{prop}[thm]{Proposition}

\newtheorem{coro}[thm]{Corollary}

\newtheorem{lem}[thm]{Lemma}
\theoremstyle{definition}
\newtheorem{defn}[thm]{Definition}
\newtheorem{example}[thm]{Example}
\newtheorem*{assumption*}{Assumption}
\theoremstyle{remark}
\newtheorem{rmk}[thm]{Remark}

\makeatletter
\def\cref@thmoptarg[#1]#2#3#4{%
    \ifhmode\unskip\unskip\par\fi%
    \normalfont%
    \trivlist%
    \let\thmheadnl\relax%
    \let\thm@swap\@gobble%
    \thm@notefont{\fontseries\mddefault\upshape}%
    \thm@headpunct{.}% add period after heading
    \thm@headsep 5\p@ plus\p@ minus\p@\relax%
    \thm@space@setup%
    #2% style overrides
    \@topsep \thm@preskip               % used by thm head
    \@topsepadd \thm@postskip           % used by \@endparenv
    \def\@tempa{#3}\ifx\@empty\@tempa%
      \def\@tempa{\@oparg{\@begintheorem{#4}{}}[]}%
    \else%
      \refstepcounter[#1]{#3}%  <<< cleveref modification
      \@namedef{cref@#3@alias}{#1}% added
      \def\@tempa{\@oparg{\@begintheorem{#4}{\csname the#3\endcsname}}[]}%
    \fi%
    \@tempa}%
\makeatother

\DeclareMathAlphabet\matheuvm{U}{zeur}{m}{n}

\newcommand{\diff}{\operatorname{d}}

\newcommand{\de}{\vec{\matheuvm{e}}}

\newcommand{\ide}{\mathring{\matheuvm{e}}}

\newcommand{\was}{\mathcal{W}_2}
\newcommand{\R}{\mathbb{R}}

\newcommand{\geoE}{\mathcal{G}eo_{\scriptscriptstyle{ \gamma(0) \in \de}} (G)}
\newcommand{\geo}{\mathcal{G}eo}
\newcommand{\ima}{\mathcal{I}m}

\newcommand{\res}{\resmes}

\renewcommand{\P}{\mathcal{P}}
\newcommand{\Leb}{{\mathscr L}}
\renewcommand{\H}{\mathcal{H}}
\newcommand{\oEdge}{\vec{E}}

\title{Wasserstein barycenters on metric graphs}
\author[1]{J\'er\^ome Bertrand\footnote{\texttt{jerome.bertrand\symbol{64}math.univ-toulouse.fr}}}
\author[1]{ Jianyu Ma \footnote{\texttt{jianyu.ma@math.univ-toulouse.fr}}}
\affil[1]{Univ Toulouse, INSA Toulouse, CNRS, IMT, Toulouse, France.}

\date{}

\begin{document}
\pagenumbering{Alph}
\begin{titlepage}
	\maketitle
	
\begin{abstract}
In this paper we solve the barycenter problem on the quadratic Wasserstein space over a possibly infinite metric graph relative to a probability measure $\mathbb{P}$. Assuming that $\mathbb{P}$ gives mass to absolutely continuous measures with respect to the 1d-Hausdorff measure on the metric graph, we prove that such a barycenter is absolutely continuous outside the set of vertices where an atomic part can exist. This result is optimal. We also show that Wasserstein barycenters are not unique in general. 
\end{abstract}

	\tableofcontents
\end{titlepage}

\pagenumbering{arabic}

\section{Introduction}

This paper is devoted to the study of Wasserstein barycenters over a metric graph $(G,d)$. By metric graph, we mean a finite or countable set of vertices connected by edges viewed as actual paths homeomorphic to compact intervals of the real line. We always assume this graph to be connected and to have finite degree at each point. We equip this graph with a length distance and its associated 1d-Hausdorff measure which turns $G$ into a proper, complete, and separable geodesic space. We refer to \Cref{sec-preli} for precise definitions.

Given a probability measure $\mathbb{P}$ on a metric space $(X,d)$, a barycenter of $\mathbb{P}$ is a minimum of the functional $z \longmapsto \int_X d^2(z,x) \diff \mathbb{P}(x)$. This condition coincides with the standard definition when $(X,d)$ is the $n$-dimensional Euclidean space.

In this paper, we consider the case where the metric space is the space of probability measures with finite second order moment $(\was(G), d_W)$ with $(G,d)$ a metric graph.

The (quadratic Kantorovitch-)Wasserstein distance on $\was(G)$ is defined by

$$d_W(\mu,\nu):= \left(\min_{\pi \in \Gamma(\mu,\nu)} \int_{X\times X} d^2(x,y) \diff \pi(x,y)\right)^{1/2},$$
where $\Gamma(\mu,\nu)$ is the set of probability measures on $G\times G$ whose marginals are $\mu$ and $\nu$ respectively. We refer to  \Cref{sec-Wass} for more details.

Existence of a barycenter on a non-proper metric space may fail, see \cite{ma-phd} for a rather complete collection of counter-examples, but barycenters over $(\was(X), d_W)$ do exist whenever the underlying metric space $(X,d)$ is proper, complete, and separable \cite{le2017existence}. Such barycenters are called \emph{Wasserstein barycenters}.

The problem of studying Wasserstein barycenters (aka the Wasserstein barycenter problem) and its regularised version have drawn a lot of interest for its many applications in statistics, image processing, and data science \cite{pereira2025}. In statistics, Wasserstein barycenter can be used as a mean on a non-linear space, and, among other things, results analogous to the law of large numbers do exist, see for instance \cite{le2017existence, bigot2018characterization}; estimating the convergence rate of empirical barycenters towards the actual barycenter is also a question of interest, the obtained rates depend on geometric properties of the space like in \cite{ahidar2020}. Wasserstein barycenters are also of primary importance in the field of metric geometry as a tool to characterise synthetic curvature bounds. Indeed to a probability measure $\mathbb{P}= (1-\lambda) \delta_{\mu_0} + \lambda \delta_{\mu_1}$ corresponds a collection of barycenters for $\lambda $ varying in $[0,1]$ which induces a geodesic curve $(\mu_{\lambda})_{\lambda \in [0,1]}$ in the Wasserstein space, and $k$-convexity of Entropy-like functionals on the Wasserstein space characterises lower Ricci curvature bound, see for instance \cite{villani2009optimal, ambrosio2014calculus, sturm2005convex, kuwada2015} for precise statements and (much) more on this subject. More precisely, barycenters and curvature bounds are related through the so-called Jensen inequality that we now report. Given a (geodesically) convex functional $F$ on a (geodesic) metric space $(X,d)$, and a measure $\mathbb{P}$ on $(X,d)$ admitting a barycenter $z \in X$, Jensen's inequality holds if
$$ F(z) \leq \int_X F(x) \diff \mathbb{P}(x).$$

Inequalities of this kind have been proved in various geometric setups: when $(X,d)$ is an Aleksandrov space of curvature bounded below \cite{paris2020} or when $(X,d)$ is a quadratic Wasserstein space over Euclidean space, a (compact) Riemannian manifold or a metric measure space with synthetic Ricci curvature bounded below (aka RCD spaces) \cite{agueh2011barycenters,kim2017wasserstein, ma2023absolute, han2024geometry}.

Uniqueness of the barycenter is false in general. A classical example is that of $\mathbb{P}=1/2(\delta_x + \delta_{-x})$ where $x$ and $-x$ are two antipodal points on a unit sphere equipped with its canonical metric $(X,d)=(\mathbb{S}^2, can)$, in that situation any point of the equator relative to $x$ and $-x$ is a barycenter of $\mathbb{P}$. Uniqueness of the barycenter holds if the space $(X,d)$ is a metric space of nonpositive curvature thanks to strong convexity properties satisfied by $(x,y) \longmapsto d(x,y)$ \cite{sturm2003probability}.  For Wasserstein barycenters, while nonnegative curvature bound is preserved from the base space $(X,d)$ to its Wasserstein space $(\was(X),d_W)$, this is \emph{not} the case for nonpositive curvature bound despite some nonpositive curvature like properties that do hold on $(\was(X), d_W)$ when $(X,d)$ is a nonpositively curved metric space \cite{villani2009optimal, bertrand2012geometric}. However, another fruitful approach is available: if the probability measure $\mathbb{P}$ gives mass to measures $\mu$ in $\was(X)$ such that for any $\nu \in \was(X)$, the distance $d_W(\mu,\nu)$ is induced by a transport map:
\begin{equation}\label{eq-trans-was}
 d_W(\mu,\nu)^2 =\int d^2(x,T(x)) \diff \mu(x),
\end{equation}
where $T:X \longrightarrow X$ satisfies $T_{\sharp}\mu= \nu$, then the Wasserstein barycenter of $\mathbb{P}$ is \emph{unique}, see for instance \cite{Santambrogio2015, ma-phd} for a detailed proof. The property appearing in \eqref{eq-trans-was} is known to holds whenever the measure $\mu$ is absolutely continuous with respect to a reference measure for metric measure spaces satisfying some regularity properties. Originally this was proved by Brenier in the Euclidean space equipped with the Lebesgue measure, then Brenier's result was extended to very general settings, see for instance \cite{mccann2001polar, bertrand2008existence, gigli2016optimal}. Besides, Wasserstein space $(\was(X),d_W)$ is known to admit a tangent space $T_{\mu} \was(X)$, which can be described in analytical or geometrical terms, at any absolutely continuous measure $\mu$ when $(X,d)$ is sufficiently regular; this was used at the formal level by Otto \cite{otto2001}, see \cite{ambrosio2008gradient} for a proof in the Euclidean case. Consequently, it is natural to seek for conditions under which a Wasserstein barycenter is absolutely continuous with respect to the reference measure, this barycenter is then said to be \emph{regular}.

The first breakthrough regarding the regularity of Wasserstein barycenters is a paper by Agueh and Carlier who proved the result in Euclidean space \cite{agueh2011barycenters}. This result was then generalised to compact manifolds \cite{kim2017wasserstein}, compact Aleksandrov spaces (of curvature bounded below) \cite{jiang2017absolute}, and possibly non-compact Riemannian manifolds with Ricci curvature bounded below \cite{ma2023absolute}. There is also a recent preprint \cite{han2024geometry} where the authors prove the absolute continuity of Wasserstein barycenter provided that $\mathbb{P}$-a.e. surely $Ent_{m}(\mu) < \infty$\footnote{this condition is probably not necessary for the barycenter being a.c.} and the base space is a metric measure space $(X,d ,m)$ with synthetic Ricci curvature bounded below ($(X,d,{m})$ is a RCD space). It is important to keep in mind that none of these results applies to metric graphs because it is a branching space and/or the curvature at each vertex is $-\infty$ (in other terms, a metric graph is \emph{not} an Aleksandrov space of curvature bounded below not even a metric measure space where the Entropy functional is $k$-convex (aka $CD(k,\infty)$ spaces) as proved in \cite{erbar2022gradient}).
Indeed, in the case of a metric (measure) graph, actually even on a metric tree, it is known that the absolute continuity of the barycenter with respect to the 1d-Hausdorff measure cannot hold in general. We provide an example from \cite{ma-phd} inspired by \cite{hotz2013sticky}. Consider the tripod in \Cref{fig:tripod_example}, formed by three copies of the unit interval $[0,1]$ joined at a common origin $o$.
\begin{figure}[h]
	\centering
	\includegraphics[page=2, scale=0.6]{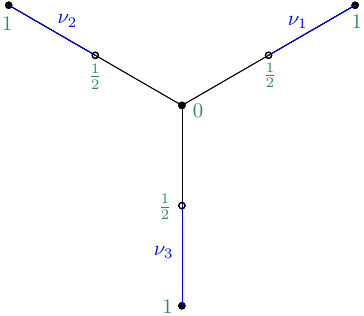}
	\caption{\label{fig:tripod_example}$\mathbb{P} = \sum_{i=1}^3 \frac{1}{3}\,\delta_{\nu_i}$ on the tripod}
\end{figure}
Let $\mathbb{P} := \frac{1}{3}\sum_{i=1}^3 \delta_{\nu_i}$ be a probability measure on the Wasserstein space over the tripod, where each $\nu_i$ is an absolutely continuous measure supported on the outer half $[\frac{1}{2}, 1]$ of a distinct branch. The unique Wasserstein barycenter of $\mathbb{P}$ is the Dirac measure $\mu_\mathbb{P} = \delta_0$ at the central vertex.

 More in general, the absence of synthetic lower curvature bounds on metric (measure) graphs forces one to tackle each standard question in  optimal mass transport with a specific approach. The Monge problem on metric graphs (where $d_W(\mu,\nu):= \min_{\pi \in \Gamma(\mu,\nu)} \int_{X\times X} d(x,y) \diff \pi(x,y)$ is used instead of the distance we mentioned above) has been studied in \cite{mazon2015optimal}. In \cite{erbar2022gradient}, the authors prove that the Wasserstein distance on a finite metric graph can be described in a more fluid dynamics framework known as Benamou-Brenier formula. We refer to \cite{benamoub2000, ambrosio2021lectures} for more on the Benamou-Brenier formula when the base space is the Euclidean space.

Our main result concerns the quasi-regularity of Wasserstein barycenters over a metric graph. This result is optimal in view of the example described in Figure \ref{fig:tripod_example}.

\begin{thm}\label{thm-main}
	Let $(G,d, \mathcal{H})$ be a metric measure graph. 
	Consider a measure $\mathbb{P} \in \mathcal{W}_2(\mathcal{W}_2(G))$ that
	gives mass to the set $\was^{ac}(G)$ of absolutely continuous measures on $(G,d)$ with respect to $\mathcal{H}$, namely $\mathbb{P}(\was^{ac}(G))>0$.
	If $\mu_{\mathbb{P}}$ is a barycenter of $\mathbb{P}$ then the restriction of $\mu_\mathbb{P}$
	to the interior of any given edge is absolutely continuous.
	Therefore, if $\mu_\mathbb{P}$ is not absolutely continuous
	then its singular part is a sum of Dirac measures located at vertices of $(G,d)$.
\end{thm}

\begin{rmk}
 A similar result for metric trees was obtained earlier by the second author in his PhD \cite{ma-phd}.
\end{rmk}

One can ask whether the barycenter of $\mathbb{P}$ is, to some extent, unique. This is not the case in general.

\begin{prop}\label{prop-non-unique} In general, given a metric graph $(G,d)$ and $\mathbb{P} \in \mathcal{W}_2(\mathcal{W}_2(G))$,  neither the atomic part nor the absolutely continuous part of a Wasserstein barycenter $\mu_{\mathbb{P}}$ is \emph{unique}.
\end{prop}

\medskip

\noindent{ \bf Organisation of the paper:} In \Cref{sec-preli}, we first define the metric graphs to be considered in this paper. We also introduce the 1d-Hausdorff measure on such a graph as the canonical reference measure. We briefly recall and/or prove the main properties of metric graph. In the second part, we recall classical properties of Wasserstein spaces, putting the emphasis on the measurable selection properties available on this space. Finally, we describe extra properties in the case of the Wasserstein space over the real line. In this very specific situation, the Wasserstein space $\was(\R)$ is \emph{flat}, namely isometric to a convex subset of a Hilbert space. Through this isometry, we prove the folklore result that the barycenter of a measure  is the standard average of the measures with respect to the given probability measure $\mathbb{P}$. We refer to \Cref{lem:Wasserstein_barycenter_real_line} for a statement.

In \Cref{sec-fine}, taking advantage of the tools introduced in \Cref{sec-preli}, we first prove the restriction property for barycenters. Generalising the fact that the restriction of an optimal plan remains optimal, \Cref{prop-restri-barycenters} is given in a broad set-up since we believe it can be useful in other contexts. This statement will allow us to reduce the proof of our main theorem to the ``mesoscopic" scale where a Wasserstein barycenter is assumed to be supported in a edge. In the last part, we also define our branched covering map and then prove its main properties that will allow us to somehow transfer the Wasserstein barycenter problem over $(G,d)$ to one such problem over the real line. It is to be emphasized that the presence of cyles in $(G,d)$ does not prevent the branched covering map to have these powerful features. 

Finally, in \Cref{sec-main} we prove our main results. The main theorem proof is in two steps as mentioned above. First, we prove the result in the case when the Wasserstein barycenter is supported in an edge, then building on the restriction property, we prove our main result holds without the restriction on the support. In the last paragraph, we discuss the non-uniqueness of the Wasserstein barycenter on a metric graph.

\section{Preliminaries}\label{sec-preli}

\subsection{Metric (measure) graphs}

In this subsection, we present a constructive definition of a metric graph
in terms of length functions defined on its edges.
This construction induces a canonical measure on the metric graph,
which coincides with the Lebesgue measure when restricted to each edge.
A metric graph equipped with this canonical measure is referred to as a \emph{metric measure graph},
a basic concept that underlies much of the subsequent development in this chapter.

\begin{defn}[Graph]\label{def-graph}

A \emph{graph} is a pair $(V,\mathcal{E})$ where $V$ denotes the set of \emph{vertices} and \\$\mathcal{E} \,\subset \{\{x,y\}; x,y \in V, x\neq y\}$ the set of (non-oriented) edges. We assume that $V$ is at most countable, hence so is $\mathcal{E}$. According to this definition, there is at most \emph{one} edge between two given distinct vertices and \emph{no} edge between a single vertex (such an edge is often called loop in graph theory). We shall use the notation $e$ or $\{x,y\}$ (with $x,y \in V, x\neq y$) for an edge. When $e=\{x,y\} \in \mathcal{E}$, $x$ and $y$ are the \emph{endpoints} of $e$ while $e$ is an edge at $x$ (or $y$). The \emph{degree} of a vertex is the number of edges at the vertex.

The set of \emph{oriented edges} is denoted by $E \subset (V \times V) \setminus \{(v,v), v\in V\}$, an oriented edge is denoted either by $\de$ or $(x,y)$. A \emph{path} $p$ in a graph is a finite sequence $p=g_0\cdots g_k$ where $k \geq 1$ and for each $i\in\{0,\cdots, k-1\}$, $\{g_i, g_{i+1}\} \in \mathcal{E}$; an edge is a particular path and we call $g_0$ and $g_k$ the \emph{ends} of the path $p$. A cycle is a path whose endpoints coincide. A graph is said to be \emph{connected} if there exists a path between any pair of distinct vertices.

\centerline{ \emph{We shall assume that the degree at each vertex is finite and at least $1$.}}
Such a graph is said to be locally finite.
\end{defn}

We now introduce metric features on a graph through the notion of length structure on the set of edges.

\begin{defn}[Metric graph]\label{defn:metric_graph}
	A \emph{length function} on a graph $(V,\mathcal{E})$ is a function $l: \mathcal{E} \rightarrow [0,\infty)$
	uniformly bounded from below by a strictly positive number,
	i.e.
	
	\begin{equation}\label{eq-length-lb}
	c:=\inf_{\alpha \in \mathcal{E}} l(\alpha) > 0.
	\end{equation}

	We identify an oriented edge $\de \in E$ with the interval $]0,l(e)[=\{x \in \R; 0<x<l(e)\}$\footnote{We use french notation to avoid confusion with an element of $\R^2$.}. When $(\de_0,\de_1)=\de \in E$, the vertices $\de_0$ and $\de_1$ are identified with the points $0$ and $l(e)$ of $[0,l(e)]$.
	
We get a \emph{metric graph} by considering the quotient space:
$$ G := \left(\coprod_{\de \in E} [0,l(e)] \right){/} \sim,$$
where points in $G$ corresponding to the same vertex of $(V, \mathcal{E})$ are identified.

Last, we define $\ide:=]0,l(e)[$ the interior of the edge $e$.
\end{defn}

Our next goal is to introduce  a (length) metric on $G$ as follows: we first define the length of a path $p=g_0\cdots g_k$ as 
$$ l(p)= \sum_{i=0}^{k-1} l(\{g_{i},g_{i+1}\}),$$
so that this length coincides with the length function introduced above for paths which are edges ({\it i.e.} when $k=1$). Note also the bound

$$ l(p) \geq k\, c>0.$$

According to \cite[\S 3.2.2]{burago2001course}, we need to define the \emph{length} of a path in $G$ with arbitrary endpoints. Let $x\neq y \in G$ and assume that $x \not\in V$ then $x \in ]0,l(e)[$ for a unique oriented edge $\de$. Similarly, if $y \not\in V$ then $y \in ]0,l(f)[$ for a unique $f\in E$. For $x,y \in e$ (equivalently $e=f$), we define a path $p$ in $e$ from $x$ to $y$ as the affine path $t \longmapsto (1-t)x +t y$ and its length is $l(p)=|x-y|$. In the remaining cases, a path from $x$ to $y$  is a finite sequence $p= x\de_i g_1\cdots g_{k-1} \vec{\matheuvm{f}}_j y$ where $i,j\in\{0,1\}$, $\de_i g_1\cdots g_{k-1} \vec{\matheuvm{f}}_j$ is a path between two vertices of $(V,\mathcal{E})$, and
$$l(p) = |i \times l(e) -x| + \sum_{i=0}^{k-1} l(\{g_{i},g_{i+1}\})	+ |j \times l(f) -y|,$$
where $g_0=\de_i$ and $g_k=\vec{\matheuvm{f}}_j$. For short, we write 
$$p: \, x \curvearrowright y$$
for a path from $x$ to $y$.

Note that the length function coincides with the previous definition when $x =\de_i$ and $y=\vec{\matheuvm{f}}_j$.

\begin{defn}[length distance on $G$] The distance $d: G \times G \longrightarrow [0,\infty]$ is
defined by

$$ d(x,y) = \left\{ \begin{array}{cl} 0 & \mbox{ if }  x=y \\ 
\infty & \mbox{ if there is no path from } x \mbox{ to } y \\
\inf_{p: \, x \curvearrowright y}  l(p)  & \mbox{ otherwise}
\end{array}\right..$$
\end{defn}

\begin{prop} Let $G$ be a connected, locally finite metric graph. Then the function $d$ turns $G$ into a metric space.
\end{prop}

We refer to \cite[Chapter 2]{burago2001course} or \cite{ma-phd}  for details. We only mention that the metric graph is originally equipped with the quotient topology, the triangle inequality follows the additivity of the length function applied to the concatenation of two paths while the positivity of $d$ is a consequence of \eqref{eq-length-lb}.

\medskip

\noindent \fbox{%
\begin{minipage}{\textwidth}
In what follows, we always assume the graph to be connected and locally finite, and we simply denote the metric graph by $(G,d)$.
\end{minipage}
}

\subsubsection*{Basic metric properties}% of metric graphs}

We first fix notation for general metric space $(X,d)$.

\begin{defn}[Various properties of metric spaces]\label{def-var-prop-ms}
A metric space $(X,d)$ is said to be
\begin{itemize}
\item  \emph{Polish} if it is complete and separable,
\item  \emph{proper} if any closed ball in $(X,d)$ is compact,\\
the closed ball of radius $R$ centered at $x \in X$ is denoted by $\widebar{B}(x,R)$,

\item \emph{geodesic} if for any pair of distinct points $x\neq y \in X$, there exists a \emph{geodesic} $\sigma$ from $x$ to $y$.\\
A geodesic $\sigma$ in $X$ is a map $\sigma : [0,1] \longrightarrow X$ such that there exists a number $l(\sigma) \geq 0$ called its \emph{length}, and, for any $s \leq t \in [0,1]$,
$$ d(\sigma(s),\sigma (t)) = l(\sigma) (t-s).$$
In particular $d(\sigma(0),\sigma(1)) = l(\sigma).$\footnote{a geodesic according to this definition is called ``minimising geodesic" in Riemannian geometry.}
 \end{itemize}

\end{defn}

We first recall that a metric graph is a particular instance of (proper) Polish space. This feature makes available most of the classical tools from the optimal mass transport theory.

\begin{thm}
	\label{thm:metric_properties_metric_graphs}
	Metric graphs are proper, complete, and geodesic metric spaces.
\end{thm}

\begin{proof}
Let us show that each closed ball $\widebar{B}(v,R)$ is compact. Indeed according to \eqref{eq-length-lb}, a path starting at $v$ whose length is at most, say, $R+1$, is contained in the union of at most  $(R+1)/c$ edges. The degree of each vertex being finite, only finitely many vertices belong to that ball. Finally, $\widebar{B}(v,R)$ is contained in (the image of) finitely many segments $[0,l(e_i)]$ for $i \in \{1,\cdots,N\}$, thus it is compact. The rest of the proof is classical, we refer to \cite{burago2001course} and the references therein for more details.
\end{proof}

Our next goal is to define a canonical measure on a metric graph. Before that, we emphasize some features of the distance on $G$ when restricted to an edge.

\begin{lem}\label{lem:edge_local_isometry}

Given $(G,d)$ a metric graph, let $\de \in E$ and $x,y \in [0,l(e)]$. Then,
\begin{itemize}

\item $ d(x,y) \leq |x-y|$,

\item $d(x,y) = |x-y|$ if $x$ and $y$ are close enough,

\item if $d(\de_0,\de_1)=l(e)$ the edge $e$ is said to be \emph{minimising}. When it is so, for any $s,t \in [0,l(e)]$,
$$ d(s,t) =|s-t|.$$
\end{itemize}
\end{lem}

\begin{proof}
The proof of the first inequality follows from the definition since the path from $x$ to $y$  contained in $e$ is the affine path of length $|x-y|$. For $x,y$ close enough, {\it i.e.} such that $|x-y| <c$ (where $c$ is as in \eqref{eq-length-lb}), any other path $\sigma$ from $x$ to $y$ has to leave $e$ and go through another edge of length at least $c$ before reentering $e$, such a path cannot be of minimal length. Note that without loss of generality, we can assume that $\sigma$ reenters $\de$ through the other end, otherwise it is clearly not a path of minimal length.
 If we further assume $\de$ is minimising then any path $\sigma$ from $x$ to $y$ as 
above can be restricted to a subinterval for which its endpoints are $\de_0$ and $\de_1$. Thus the length of $\sigma$ is greater than $d(\de_0, \de_1) =l(e) > |x-y|$  whenever $\{x,y\} \neq \{0,l(e)\}$.
\end{proof}

The canonical measure $\mathcal{H}$ on $(G,d)$ can be defined using its edges and vertices, it is indeed the one-dimensional Hausdorff measure on $(G,d)$ \cite[\S 1.7]{burago2001course}.

\begin{defn}[Canonical measures on metric graphs]
	\label{defn:canonical_measure}
	Let $G = (V, E, d)$ be a metric graph.
	The canonical measure $\mathcal{H}$ is the $1d$-Hausdorff measure induced by the distance $d$ on $G$. Thanks to  \Cref{lem:edge_local_isometry}, the measure $\mathcal{H}$ 	\emph{gives no mass to the set of vertices} $V$, and for each
	oriented edge $\de \in E$, the  restriction $\mathcal{H} \res_{\de}$ to $\de$  is the Lebesgue measure on $[0, l(\de)]$.
\end{defn}

\begin{proof}
By \Cref{lem:edge_local_isometry}, for any $\de \in E$, any $x \leq y \in [0,l(e)]$,  $\H^1([x,y]) \leq |x-y|$ and equality holds whenever $d(x,y)=|x-y|$. Since the index at each vertex is finite, this implies $\H(\{v\})=0$ for any $v \in V$. Last, since the $1d$-Hausdorff measure with respect to $|\cdot|$ on $[0,l(e)]$ coincides with the Lebesgue measure (up to scaling) \cite{evans2018measure}, the second item of \Cref{lem:edge_local_isometry} together with simple covering arguments taking advantage of the compactness of $[0,l(e)]$, lead to the fact that $\mathcal{H}\res_{\de}$ coincides with the Lebesgue measure on $[0,l(e)]$.
\end{proof}

%%%%%%%%%%%%%%%%%%%%%%%%%%%%%%%%%%%%%%%%%%%%%%%%%%%%%%%%%%%%%%%%%%%%%%%%%%%%%%%%%%%%

In the course of proving our main result, we shall need to single out some special points which are, in some weak sense, in the Cut Locus of the graph. We define the exact meaning of this in the next definition.

\begin{defn}\label{def:cut-loc} Let $(G,d)$ be a metric graph. For $x\neq y \in G, x \not\in V$, there exist $e,f \in E$ such that
$x \in \ide$ and $y \in f$. 

If there are at least two distinct geodesics $\gamma_1, \gamma_2$  between $x$ and $y$ such that $\gamma_1$ and $\gamma_2$ exit
$e$ by different endpoints, the pair $(x,y)$ is said to belong to the \emph{oriented} cut locus $\overrightarrow{Cut}(G)$. In particular this forces $e \neq f$. %Similarly, we define $\overleftarrow{Cut}(G)$ as the set of $(x,y) $ such that $(y,x) \in \overrightarrow{Cut}(G)$.

%If there are at least two distinct geodesics $\gamma_1, \gamma_2$ %(resp $\sigma_1, \sigma_2$) 
% between $x$ and $y$ such that $\gamma_1$ and $\gamma_2$ leaves
%$e$ and $f$ by different endpoints, the pair $(x,y)$ is said to belong to the cut %locus $Cut(G)$.

\end{defn}

%\begin{rmk}
%Note that the fact that $(x,y)$ and $(y,x)$ belong to $\overrightarrow{Cut}$ does not imply in general that $(x,y) \in Cut(G)$. 
%
%\end{rmk}

With that definition in hand, one can easily generalise part of a classical fact in Riemannian geometry to our metric graph set-up.

\begin{lem}\label{lem:semi-ext-geod}
Let $(G,d)$ be a metric graph and $x\neq y$ such that $(x,y) \in (G\setminus V \times G)\setminus \overrightarrow{Cut}(G)$. Then, any geodesic $\gamma$ from $y$ to $x$ can be locally extended to a geodesic beyond $x$.

\end{lem}

\begin{proof}
The proof is by contradiction. Assume $\gamma$ cannot be extended beyond $x \in \ide$. Identifying $e$ with the segment $[0,l(\de)]$, so that $x \in ]0,l(e)[$, assuming that $\gamma$ reaches  $x$ by going through $\de_0$; one gets for any $\varepsilon>0$ small, the existence of geodesic from $y$ to $x +\varepsilon$ whose length is smaller than $d(x,y) + \varepsilon$, this geodesic has to reach $x+\varepsilon$ through $\de_1$. The lower semicontinuity of the length functional combined with an easy extraction argument easily provides us with another geodesic between  $x$ and $y$ entering $\de$ through $\de_1$ hence a contradiction.
\end{proof}

\begin{lem}\label{lem:countability-cut}
Let $(G,d)$ be a metric graph  and $v \in V$. Then the set

$$\overrightarrow{Cut}_v(G):= \{y \in G; (y,v) \in \overrightarrow{Cut}(G)\}$$
is discrete, in particular it is at most countable.
\end{lem}

\begin{proof}
Let $y\in \overrightarrow{Cut}_v(G)$. 
Thus, there are two geodesics $\gamma_1$ and $\gamma_2$ connecting $v$ and $y$ crossing different endpoints of $e$. As a consequence, the cardinality of $\overrightarrow{Cut}_v(G) \cap \ide$ is at most $1$. Therefore, since the number of edges is at most countable and the degree of each vertex is finite, we infer the claim.
\end{proof}

%%%%%%%%%%%%%%%%%%%%%%%%%%%%%%%%%%%%%%%%%%%%%%%%%%%%%%%%%%%%%%%%%%%%%%%%%%%%%%%%%%%%%%%%%%%%%%%%%%%%%%%%%%%%%%%%%%
%%%%%%%%%%%%%%%%%%%%%%%%%%%%%%%%%%%%%%%%%%%%%%%%%%%%%%%%%%%%%%%%%%%%%%%%%%%%%%%%%%%%%%%%%%%%%%%%%%%%%%%%%%%%%%%%%%

%%%%%%%%%%%%%%%%%%%%%%%%%%%%%%%%%%%%%%%%%%%%%%%%%%%%%%%%%%%%%%%%%%
%%%%%%%%%%%%% Measurable selection %%%%%%%%%%%%%%%%%%%%%%%%%%%%%%
%%%%%%%%%%%%%%%%%%%%%%%%%%%%%%%%%%%%%%%%%%%%%%%%%%%%%%%%%%%%%%%%%%%%

\subsection{Wasserstein spaces}\label{sec-Wass}

In this paper we consider the quadratic Wasserstein space over a metric space $(X,d)$, namely the space of Borel probability measures $\mu$ with finite second order moment:
$$ \int_X d^2(x_0,x) \diff \mu(x),$$
where $x_0$ is an arbitrary fixed point of $X$.

The set of such probability measures is denoted by $\was(X)$. The space $\was(X)$ is equipped with the Wasserstein distance whose definition is recalled below. To define this metric, let us introduce for $\mu_1,\mu_2 \in \was(X)$, 
$$ \Gamma(\mu_1,\mu_2):=\{\pi \in \mathcal{P}(X\times X); \, {p_i}_{\sharp} \pi =\mu_i, i\in\{1,2\}\}$$
the set of transport plans between $\mu_1$ and $\mu_2$. In this definition, the maps $p_i : X \times X \longrightarrow X$ are the canonical projections defined by $p_1(x,y)=x$ and $p_2(x,y)=y$.

The Wasserstein distance between two measures  $\mu,\nu \in \was(X)$ is defined as
$$d_W(\mu,\nu):= \left(\min_{\pi \in \Gamma(\mu,\nu)} \int_{X\times X} d^2(x,y) \diff \pi(x,y)\right)^{1/2}.$$

The subset of \emph{optimal} plans between $\mu$ and $\nu$ is denoted by $\Gamma^o(\mu,\nu)$; it is made of plans $\pi^o\in \Gamma(\mu,\nu)$ for which
$$ d_W(\mu,\nu)^2 = \int_{X\times X} d^2(x,y) \diff \pi^o(x,y).$$

Let us recall three classical results on Wasserstein spaces:

\begin{thm} Let $(X,d)$ be a metric space. Then,

\begin{enumerate}
\item $(\was(X),d_W)$ is Polish whenever $(X,d)$ is so,
\item $(\was(X),d_W)$ is compact whenever $(X,d)$ is so, otherwise $(\was(X),d_W)$ is \emph{not} proper,
\item $(\was(X),d_W)$ is Polish and geodesic\footnote{see \Cref{def-var-prop-ms} for the definition} whenever $(X,d)$ is so.
\end{enumerate}
\end{thm}

We refer to \cite{Santambrogio2015,ambrosio2021lectures,villani2009optimal}
 for proofs and a complete review of the standard topological and metric properties of the quadratic Wasserstein space.
 
\begin{rmk} Combining \Cref{thm:metric_properties_metric_graphs} together with the previous theorem yields that in the case of a metric graph $(G,d)$, $(\was(G),d_W)$ is a geodesic Polish space, non-proper except if $(G,d)$ is compact (then it is actually compact).
\end{rmk}

In the next part, we recall some results on measurable selection that are often useful to compare Wasserstein spaces over distinct base spaces. 

\subsubsection{Measurable selection}\label{subs-measu-selec}

In order to study the Wasserstein barycenter, we shall need to apply several \emph{measurable} selections of family of measures -like transport plans with a variable marginal for instance. 

In the result below, the product space $X^2$ is equipped with the product metric
$$d((x_1, x_2), (y_1, y_2))^2 = d(x_1, y_1)^2 + d(x_2, y_2)^2.$$

\begin{thm}\label{thm-sel-opt-plan} Let $(X,d)$ be a Polish space. There exists a measurable map
\begin{eqnarray}\label{oplan-selec}
\pi : \was(X) \times \was(X)  & \longrightarrow  &   \was(X^2), \nonumber \\
						(\mu,\nu)  & \longmapsto    &  \pi_{\mu}^{\nu}
						\end{eqnarray} 
such that $\pi_{\mu}^{\nu} \in \Gamma^o(\mu,\nu)$ for all $\mu,\nu$.
\end{thm}

We refer to \cite[Corollary 5.22]{villani2009optimal} for a proof of this result. Regarding the base space, we also have the following selection result.

\begin{prop} Let $(X,d)$ be a proper geodesic Polish metric space and let $\mathcal{G}eo(X)$ be the space of geodesics in $(X,d)$. There exists a measurable selection map:
\begin{eqnarray}\label{geod-selec}
\gamma_{\cdot,\cdot} :  X \times X  & \longrightarrow  &   \mathcal{G}eo(X), \nonumber \\
						(x,y)  & \longmapsto    &  \gamma_{x,y}
						\end{eqnarray} 
where $\gamma_{x,y} (0) =x$ and $\gamma_{x,y} (1) =y  $.

\end{prop}

The proof of this result is a rather straightforward consequence of the Arzela-Ascoli theorem. We refer to \cite{lisini2007} for a proof of it in the general setting of Polish spaces.

\begin{defn}[Evaluation map]
For $t \in [0,1]$, we define
\begin{eqnarray}
           \cdot_t :  \mathcal{G}eo(X)  &  \longrightarrow  & \;\;\;X \nonumber \\
           					\gamma  & \longmapsto & \gamma_t:= \gamma(t) 
\end{eqnarray}

\end{defn}

By combining the two previous results, we get the existence of optimal \emph{dynamical} plans. A measurable selection of such plans does exist.

\begin{thm} Let $(X,d)$ be a proper geodesic Polish space. There exists a measurable map
\begin{eqnarray}\label{odynplan-selec}
\Pi : \was(X) \times \was(X)  & \longrightarrow  &   \P(\mathcal{G}eo(X)), \nonumber \\
				\phantom{v}		(\mu,\nu)  & \longmapsto    &  \Pi_{\mu}^{\nu}
\end{eqnarray} 
such that 

$(\cdot_0, \cdot_1) _{\sharp} \Pi_{\mu}^{\nu} \in \Gamma^o(\mu,\nu)$ for all $\mu,\nu$. Besides, the following equality holds
\begin{equation}\label{eq-dyn-transp-pro}
d_W(\mu,\nu)^2 = \int_{\mathcal{G}eo(X)} \, s^2(\gamma) \, \diff \Pi_{\mu}^{\nu}(\gamma),
\end{equation} 
where $s(\gamma)=d(\gamma(0),\gamma(1))$ is the \emph{speed} (or length) of $\gamma$.
\end{thm}

\begin{proof}
The set $\mathcal{G}eo(X)$ equipped with the uniform distance
$$ d_{\infty} (\gamma, \tilde{\gamma}) = \max_{t \in [0,1]} d(\gamma(t), \tilde{\gamma}(t)),$$
is a $\sigma$-compact metric space hence complete and separable. We then denote the quadratic Wasserstein space over $(\mathcal{G}eo(X), d_{\infty})$ by $(\was(\mathcal{G}eo(X)), d_W)$. Observe that for a probability measure $\Pi$ over $\mathcal{G}eo(X)$, $\Pi \in \was(\mathcal{G}eo(X))$ amounts to
$$ \int s^2(\gamma) \diff \Pi(\gamma)<\infty,$$
thanks to the bound, say for a constant geodesic $\gamma_0\equiv x_0$, 
$$d_{\infty}^2(\gamma_0,\gamma) \leq 2(d^2(x_0, \gamma(0)) + s^2(\gamma)).$$

The proof then follows from the Kuratowski and Ryll-Nardzewski measurable selection theorem \cite{bogachev2007measure}. We consider the set-valued mapping:
$$ F : \was(X) \times \was(X)   \rightrightarrows     \was(\mathcal{G}eo(X)), $$
where $F(\mu,\nu) ={\bf \Gamma}_{dyn}^o (\mu,\nu)$ is the set of optimal dynamical plans whose ends are $\mu$ and $\nu$ (in this order). We have to check that $F(\mu,\nu)$ is a non-empty closed subset of  $\was(\mathcal{G}eo(X))$ (we shall actually prove the \emph{compactness} of $F(\mu,\nu)$) and that for any closed subset $K$ of 
$\was(\mathcal{G}eo(X))$, the set
$$ \mathcal{S}_K:= \{(\mu,\nu) \in \was(X)^2; F(\mu,\nu) \cap K \neq\emptyset \}$$
is a measurable subset of $\was(X)^2\;$\footnote{the standard statement of this measurable theorem requires to check this property for $K$ an open subset of $\was(X)^2$, this property follows from the one we state by elementary calculations noting that any open subset of a metric space is a countable union of closed sets}.

We have already checked the existence of optimal dynamical plans between $\mu$ and $\nu$. The compactness of $F(\mu,\nu)$ is a consequence of Prokhorov's theorem. Indeed, given $R>0$  such that $\max (\mu (X\setminus  \overline{B}(x_0,R)), \nu (X\setminus  \overline{B}(x_0,R))) < \varepsilon$ then we claim that the set
$$ \mathcal{G}_R:=\{\gamma \in \mathcal{G}eo(X); \gamma(0), \gamma(1) \in \overline{B}(x_0,R)\}$$
is compact. First, the set $\mathcal{G}_R$ is clearly closed.

To prove the claim, note that for $\gamma \in \mathcal{G}_R$,  $s(\gamma)=d(\gamma(0), \gamma(1)) \leq 2R$ thus  $d(x_0, \gamma(t)) \leq R +1/2 s(\gamma) \leq 2R$ for all $t\in [0,1]$. Consequently,
$\mathcal{G}_R $  is contained in the subset  $S$ of Lipschitz functions $f :[0,1] \longrightarrow X$ whose range satisfies $f([0,1]) \subset \overline{B}(x_0,2R)$ and whose Lipschitz constant is at most $2R$. The space $(S, d_{\infty})$ is compact thanks to the Ascoli-Arzela theorem.

Now, let us prove that $\mathcal{S}_K$ is a closed subset of $\was(X)^2$. Let $$(\mu_n,\nu_n) \underset{ n \rightarrow \infty}{\longrightarrow} (\mu_{\infty}, \nu_{\infty})$$ be a converging sequence where the $(\mu_n,\nu_n)$ belong to $\mathcal{S}_K$. For each $n$, let $\Pi_n \in F(\mu_n,\nu_n) \cap K$. Since $\mu_n \rightarrow \mu_{\infty}$ and $\nu_n \rightarrow \nu_{\infty}$ in $\was(X)$, the collection $\{(\mu_n,\nu_n); n \in \mathbb{N}\} \cup \{(\mu_{\infty},\nu_{\infty})\}$ is a compact set and the proof of the claim above leads to the compactness of $\overline{\cup_{n \in \mathbb{N}} F(\mu_n,\nu_n)}$ (this is classical for optimal plans). Finally, the continuity of the evaluation mapping $\Pi \longmapsto (e_0,e_1)_{\sharp} \Pi$ implies that any accumulation point of $(\Pi_n )_{n \in \mathbb{N}}$ belongs to $F(\mu_{\infty},\nu_{\infty}) \cap K$ and the result is proven.

%any $\Pi \in F(\mu,\nu)$ has at least $1-2\varepsilon$ amount of mass concentrated in the compact subset of geodesics of speed at most $\rm{diam} (C)$ and lying within a closed ball centered at $x_0$ and of radius $ \max_{ y \in C} d(x_0,y) +  \rm{diam} (C)$. Besides a weak limit of optimal dynamical plans has to be optimal since $\gamma \longmapsto s(\gamma)^2 \geq 0$ is continuous. 
 
%Last, given $(\mu_n,\nu_n)_{n \in \mathbb{N}} \in \mathcal{S}^{\mathbb{N}}$ a converging sequence, the limit is still in $\mathcal{S}$ as a consequence of the compactness of $K$ and Prokhorov's theorem. Therefore $\mathcal{S}$ is a closed subset and the result is proved.

\end{proof}

In the case of a metric graph, we shall take advantage of the graph structure in order to further study dynamical plans. This will be done in \Cref{sec-fine}.

%%%%%%%%%%%%%%%%%%%%%%%%%%%%%%%%%%%%%%%%%%%%%%%%%
%%%%%%%%%%%% Barycenter %%%%%%%%%%%%%%%%%%%%%%%
%%%%%%%%%%%%%%%%%%%%%%%%%%%%%%%%%%%%%%%%%%%%%%%%%%
\subsubsection{Barycenter}

We end this introductive part on Wasserstein spaces with our main object of interest: Wasserstein barycenter. We start with the definition of a barycenter on an arbitrary metric space before briefly discussing the basic properties of barycenters over Wasserstein space.

\begin{defn}[Barycenter]
	\label{defn:barycenter}
	Let $(X,d)$ be a metric space and let $\mu$ be a Borel probability measure on $X$ such that
	$\int_{X} d(x_0, y)^2 \diff \mu(y) < \infty$ for some point $x_0 \in E$.
	We call $z_\mu \in X$ a barycenter of $\mu$ if
	\[
		\int_{X} d(z_\mu, y)^2 \diff \mu(y)
		= \min_{x \in X} \int_{X} d(x, y)^2 \diff \mu(y).
	\]
\end{defn}

Barycenters always exist in proper spaces since a minimizing sequence is
bounded and thus pre-compact. We refer to Ohta \cite{ohta2012barycenters} for more details and
some other properties  of barycenters in a proper space.

However, simple counter-examples show the lack of uniqueness for barycenters.

\begin{example}
	\label{example:sphere_barycenters_equator}
	Let $\mathbb{S}^2$ be the two-dimensional sphere
	and fix two antipodal points $x, -x$ of $\mathbb{S}^2$.
	Consider the measure $\mu = \frac{1}{2} \delta_x + \frac{1}{2} \delta_{-x}$ on $\mathbb{S}^2$.
	Then all points in the equator, i.e., the set of all points with equal distances to $x$ and $-x$,
	are barycenters of $\mu$.
\end{example}

Barycenters on non-proper metric spaces may not exist, see \cite{ma-phd} for counter-examples. Recall that $(\was(X), d_W)$ is \emph{not} proper unless $(X,d)$ is compact. The existence of Wasserstein barycenter is therefore not straightforward. Nevertheless, a general existence result is proved in \cite{le2017existence}. 

\begin{thm}If $(X,d)$ is a locally compact, complete, geodesic space then any measure $\mathbb{P} \in \was(\was(X))$ has a barycenter.

\end{thm}
%%%%%%%%%%%%%%%%%%%%%%%%%%%%%%%%%%%%%%%%%%%%%%%%%%%%%%%%%%%%%%%%%%%%%%%%%%%%%%%%%
%%%%%%%%%%%%%%%%%%% wass (\R) %%%%%%%%%%%%%%%%%%%%%%%%%%%%%%%%%%%%%%%%%%%%%
%%%%%%%%%%%%%%%%%%%%%%%%%%%%%%%%%%%%%%%%%%%%%%%%%%%%%%%%%%%%%%%%%%%%%%%%%%%%%

We shall need finer properties of the Wasserstein space when the base space is $(\R, |\cdot|)$ that are explained in the next paragraph.

\subsection{Wasserstein space over the real line}

Most results in this section are based on the fact that $\was(\R)$ is a flat space, a precise statement is \Cref{thm:optimal_transport_on_real_line}. The main object to express the extra properties of $\was(\R)$ is the quantile function of $\mu \in \was(\R)$ which is discussed in the next part.

\subsubsection{Quantile functions}

In the literature, there exist different definitions of
distribution functions and quantile functions, we thus make our definition explicit.

\begin{defn}[Distribution functions and quantile functions]
	\label{defn:quantile_function}
	Let $\mu$ be a probability measure on $\mathbb{R}$.
	Its distribution function $f_\mu: \mathbb{R} \rightarrow [0, 1]$
	is defined by $f_\mu (x): = \mu((-\infty, x])$,
	and its quantile function $f_\mu^{-1}: [0, 1] \rightarrow \widebar {\mathbb{R}}$
	is defined by
	\begin{align*}
		f_\mu^{-1}(t) & := \inf_x \{x \in \mathbb{R} ; f_\mu(x) > t \}
		\text{ for } 0 < t < 1                                            \\\text{ and }\qquad
		f_\mu^{-1}(0) & := \lim_{t \downarrow 0} f_\mu^{-1}(t),\quad
		f_\mu^{-1}(1) := \lim_{t \uparrow 1} f_\mu^{-1}(t),
	\end{align*}
	where $\widebar {\mathbb{R}}=[-\infty,\infty]$.
\end{defn}
%Note that in the above definition, we make use of the convention
%\begin{equation}
%	\label{equa:infimum_of_empty_set}
%	\inf_{x \in (y, z) } \emptyset = z, \quad
%\end{equation}
%for $y,z \in \widebar{\R}$.

With this definition, the quantile function is fully characterised by its values on $(0,1)$. While the values at the boundary of $[0,1]$ can be infinite, the quantile function is real-valued as a function on $(0,1)$. Moreover, as the distribution function $f_{\mu}$, the quantile function $f_\mu^{-1}$ is a non-decreasing, right continuous function on $(0,1)$. Besides, any $g:(0,1)\longrightarrow \R$ non-decreasing, right continuous function is the quantile function of some probability $\mu \in \mathcal{P}(\R)$. We refer to \cite{bobkovledoux2019, villani2021topics} and the references therein for more on this properties.

In the next lemma, to be used in the proof of our main result, the size of the support of $\mu$ is related to the range of its quantile function.

\begin{lem}
	\label{lem:quantile_function_range_and_measure_support}
	Let $\mu$ be a probability measure on $\mathbb{R}$.
	The infimum and supremum of the support of $\mu$
	are related to its quantile function as follows,
\[
	f_\mu^{-1}(0) = \inf \operatorname{supp}(\mu)
	\quad \text{ and } \quad
	f_\mu^{-1}(1) = \sup \operatorname{supp}(\mu).
\]
	In particular,
	$\mu$ has compact support if and only if $f_\mu^{-1}$ is finite on the whole unit interval $[0, 1]$.
\end{lem}

\begin{proof}
	We first prove the following two inequalities,
\[
	f_\mu^{-1}(0) \le \inf \operatorname{supp}(\mu)
	\quad \text{ and } \quad
	f_\mu^{-1}(1) \ge \sup \operatorname{supp}(\mu).
\]
	The case that $f_{\mu}^{-1}(0) = - \infty$ or $f_{\mu}^{-1}(1) = + \infty$
	is trivial, we are left to consider the case where they are finite.
	If $y < f_\mu^{-1}(0)$, then for any $t \in (0, 1)$,
	$\frac{1}{2} y + \frac{1}{2}f_\mu^{-1}(0) < f_\mu^{-1}(t)$
	and thus $\frac{1}{2} y + \frac{1}{2}f_\mu^{-1}(0) \notin \{ x \in \mathbb{R} \mid f_\mu(x) > t \}$.
	It follows that $f_\mu(\frac{1}{2} y + \frac{1}{2}f_\mu^{-1}(0)) = 0$,
	and hence the point $y$, strictly smaller than $\frac{1}{2} y + \frac{1}{2}f_\mu^{-1}(0)$,
	is not in the support of $\mu$.
	As $y$ is arbitrarily chosen,
	we have $(-\infty, f_{\mu}^{-1}(0)) \cap \operatorname{supp}(\mu) = \emptyset$
	and thus $f_\mu^{-1}(0) \le \inf \operatorname{supp}(\mu)$.
	If $z > f_\mu^{-1}(1)$, then for any $t \in (0, 1)$,
	$\frac{1}{2} z + \frac{1}{2}f_\mu^{-1}(1) > f_\mu^{-1}(t)$
	and thus $\frac{1}{2} z + \frac{1}{2}f_\mu^{-1}(1) \in \{ x \in \mathbb{R} \mid f_\mu(x) > t \}$.
	It follows that $f_\mu(\frac{1}{2} z + \frac{1}{2}f_\mu^{-1}(1)) = 1$,
	and hence the point $z$, strictly bigger than $\frac{1}{2} z + \frac{1}{2}f_\mu^{-1}(1)$,
	is not in the support of $\mu$.
	As $z$ is arbitrarily chosen,
	we have $(f_{\mu}^{-1}(1), + \infty) \cap \,\operatorname{supp}(\mu) = \emptyset$
	and thus $f_\mu^{-1}(1) \ge \sup \operatorname{supp}(\mu)$.

	We now prove the inequalities,
\[
	f_\mu^{-1}(0) \ge \inf \operatorname{supp}(\mu)
	\quad \text{ and } \quad
	f_\mu^{-1}(1) \le \sup \operatorname{supp}(\mu).
\]
	The case that $\inf \operatorname{supp}(\mu) = - \infty$ or $\sup \operatorname{supp}(\mu) = + \infty$
	is trivial, we are left to consider the case where they are finite.
	If $y < \inf \operatorname{supp}(\mu)$, then $f_\mu(y) = 0$ and thus $y \le f_\mu^{-1}(t)$ for all $t \in (0, 1)$,
	which implies $y \le f_\mu^{-1}(0)$.
	As $y$ is arbitrarily chosen, we must have 
	$f_\mu^{-1}(0) \ge \inf \operatorname{supp}(\mu)$.
	If $z > \sup \operatorname{supp} (\mu)$, then $f_\mu(z) = 1$ and thus $z \ge f_\mu^{-1}(t)$ for all $t \in (0, 1)$,
	which implies $z \ge f_\mu^{-1}(1)$.
	As $z$ is arbitrarily chosen, we must have 
	$f_\mu^{-1}(1) \le \sup \operatorname{supp}(\mu)$.

	Since $\mu$ has compact support if and only if both
	$\inf \operatorname{supp}(\mu)$ and $\sup \operatorname{supp}(\mu)$ are finite,
	our last statement in the lemma follows.
\end{proof}

In Subsection \ref{sec:quantile}, we describe the barycenter $\mu_{\mathbb{P}}$ of a probability measure $\mathbb{P} \in \was(\was(\R))$ as the integral of the map $\mu \longmapsto f_{\mu}^{-1}(t)$ with respect to $\mathbb{P}$. The goal of the next lemma is to make sure that this map is measurable.

\medskip

It is a classical result that weak convergence of probability measures $\mu_n \rightharpoonup \mu_{\infty}$ can be equivalently characterised as the pointwise convergence of the corresponding distribution functions $f_{\mu_{n}}(t)$ towards $f_{\mu_{\infty}}(t)$ at every continuity point $t$ of $f_{\mu_{\infty}}$. The same characterisation holds when quantile functions are used instead, also at continuity points of $f_{\mu_{\infty}}^{-1}$; we refer to \cite[Proposition 5.7 of Chapter III]{cinlar2011probability} for a proof. The discontinuity points of a nondecreasing  function $g:(0,1) \longrightarrow \R$ being at most countable, one infer the following property:

\begin{lem}
	\label{lem:upper_semicontinuity_quantile_functions}
	Let $\{\mu_n\}_{n \in \mathbb{N}}$ be a sequence of probability measures on $\mathbb{R}$ converging weakly to $\mu_{\infty}$
	Then for any $t \in (0,1)$, it holds
	\begin{equation}
		\label{equa:upper_semicontinuity_quantile_functions}
		\limsup_{n \rightarrow +\infty} f_{\mu_n}^{-1}(t) \le f_{\mu_{\infty}}^{-1}(t),\,
		\forall\; t \in [0, 1) \quad \text{and} \quad
		\liminf_{n \rightarrow +\infty} f_{\mu_n}^{-1}(1) \ge f_{\mu_{\infty}}^{-1}(1).
	\end{equation}
\end{lem}

\begin{proof}
	For $t \in [0, 1)$, there is sequence of decreasing and positive numbers $\{\varepsilon_k\}_{k \in \mathbb{N}}$
	such that $ t + \varepsilon_k \in (0, 1)$ and $f_\mu^{-1}$ is continuous at $ t + \varepsilon_k$.
	Since quantile functions are non-decreasing, we have
	\[
		\limsup_{n \rightarrow +\infty} f_{\mu_n}^{-1}(t)  \le
		\limsup_{n \rightarrow +\infty} f_{\mu_n}^{-1}(t + \varepsilon_k)
		= \lim_{n \rightarrow +\infty} f_{\mu_n}^{-1}(t + \varepsilon_k)
		= f_{\mu}^{-1}(t + \varepsilon_k),
	\]
	which implies $\limsup_{n \rightarrow +\infty} f_{\mu_n}^{-1}(t) \le f_{\mu}^{-1}(t)$
	by the right-continuity of $f_\mu^{-1}$ at $t$.
  The proof for $t=1$ is similar and left to the reader.
	
\end{proof}

We end this part by showing that $\was(\R)$ is isometric to a closed convex subset of a Hilbert space.

In what follows, $[g]$ stands for the equivalence class in $L^2([0,1])$ of a function $g:(0,1) \longrightarrow \R$.

Let us define

$$Q:= \{[g] \in L^2([0,1]); g:(0,1) \longrightarrow \R \mbox{ nondecreasing, right continuous }\}.$$

\begin{thm}
	\label{thm:optimal_transport_on_real_line}
The map
$$\begin{array}{rcl}
\was(\R) & \longrightarrow & Q \subset L^2([0,1]) \\
  \mu & \longmapsto &  f_{\mu}^{-1}
\end{array}	$$
	
	is  a surjective isometry onto the \emph{closed convex} subset $Q$ of $L^2([0,1])$:
		\begin{equation}
		\label{equa:transport_cost_on_real_line}
		d_W(\mu, \nu)^2 = \int_0^1 [f_\mu^{-1}(t) - f_\nu^{-1}(t)]^2 \diff t.
	\end{equation}
\end{thm}

We refer to \cite[Theorem 2.18]{villani2021topics} and \cite{Santambrogio2015} for a proof.

\begin{rmk}
Note that if $\nu=\delta_0$ then $f_{\nu}^{-1}(t)\equiv 0$ on $(0,1)$. Therefore, for such a choice \eqref{equa:transport_cost_on_real_line} reads
$$ \int_{\R} |x|^2 \diff \mu(x) = \int_0^1 [f_\mu^{-1}(t)]^2 \diff t.$$

\end{rmk}

\subsubsection{Integral formula for Wasserstein barycenters on the real line}\label{sec:quantile}

In this part, we take advantage of the linear structure on $L^2([0,1])$ to prove a more standard expression of a Wasserstein barycenter on $\R$ through its quantile function.
In the case of a finitely supported probability measure $\mathbb{P} = \sum_{i=1}^n \lambda_i \delta_{\nu_i}$,
the quantile function of its barycenter $\mu_\mathbb{P}$ is given by the formula
\cite[\S 3.1.4]{panaretos2020invitation},
\[
	f_{\mu_\mathbb{P}}^{-1}(t) = \sum_{i=1}^n \lambda_i \, f_{\nu_i}^{-1}(t)
	= \int_{\mathcal{W}_2(\mathbb{R})} f_\nu^{-1}(t) \diff \mathbb{P}(\nu),\quad
	\forall\, t \in [0, 1].
\]

For general $\mathbb{P} \in \was(\was(\R))$, we have

\begin{thm}[Wasserstein barycenters on the real line]
	\label{lem:Wasserstein_barycenter_real_line}
	Let $\mathbb{P} \in \mathcal{W}_2(\mathcal{W}_2(\mathbb{R}))$ be a probability measure
	on the Wasserstein space $(\mathcal{W}_2(\mathbb{R}), d_W)$.
	Then $\mathbb{P}$ has a unique Wasserstein barycenter ${\mu_\mathbb{P}} \in \mathcal{W}_2(\mathbb{R})$,
	whose quantile functions satisfies
	\begin{equation}
		\label{equa:formula_Wasserstein_barycenter_on_R}
		f_{\mu_\mathbb{P}}^{-1}(t) =
		\int_{\mathcal{W}_2(\mathbb{R})} f_\nu^{-1}(t) \diff \mathbb{P}(\nu),
		\quad \forall\, t \in [0, 1].
	\end{equation}
	In particular, the integral in (\ref{equa:formula_Wasserstein_barycenter_on_R})
	is finite for $t \in (0, 1)$. %and the inequality still holds
%	when it takes possibly infinite values for the case $t = 0, 1$.
\end{thm}

This theorem was proved under stronger assumptions in \cite{bigot2018characterization}.

\begin{proof}
	By \Cref{lem:upper_semicontinuity_quantile_functions},
	the map $\nu \mapsto f_\nu^{-1}(t)$ is upper semi-continuous for $t \in [0, 1)$
	and lower semi-continuous for $t = 1$,
	hence it is a measurable map for any parameter  $t \in [0, 1]$. 
	Besides, it follows from \Cref{thm:optimal_transport_on_real_line} and Fubini's theorem that
	\[
		\int_0^1 \int_{\mathcal{W}_2(\mathbb{R})} [f_\nu^{-1}(t)]^2 \diff \mathbb{P}(\nu) \diff t
		= \int_{\mathcal{W}_2(\mathbb{R})} \int_0^1 [f_\nu^{-1}(t)]^2 \diff t \diff \mathbb{P}(\nu)
		= \int_{\mathcal{W}_2(\mathbb{R})}  d_W(\delta_0, \nu)^2 \diff \mathbb{P}(\nu)  < + \infty.
	\]
	
	Consequently we infer from the Cauchy–Schwarz inequality
	that the function $g: [0, 1] \rightarrow \widebar{\mathbb{R}}$ defined by $ g(t) :=
		\int_{\mathcal{W}_2(\mathbb{R})} f_\nu^{-1}(t) \diff \mathbb{P}(\nu)$ for $t \in [0, 1]$
	is an element in $L^2([0, 1])$.

	We claim that $g$ is a non-decreasing and right-continuous function
	on $(0, 1)$ with $g(0) = \lim_{t \downarrow 0} g(t)$ and $g(1) = \lim_{t \uparrow 1} g(t)$.
	We first show that $g$ must be finite on $(0, 1)$.
	Indeed, if $g(t) = + \infty$ for some $t \in (0, 1)$, then
	$g(s) = + \infty$ for any $ s \in [t, 1]$ since any quantile function $f_\nu^{-1}$
	is non-decreasing, which contradicts the fact that $g \in L^2([0,1])$.
	Due to the same reason, we cannot have $g(t) = - \infty$ for some $ t \in (0, 1)$.
	Hence, $g$ is finite and non-decreasing on $(0, 1)$.
	Fix $t \in [0, 1)$, we show that $g$ is right-continuous at $t$.
	Let $\{t_n\}_{n \ge 1} \subset (t, \frac{1+t}{2})$ be a decreasing sequence
	smaller than $\frac{1+t}{2}$ that converges to $t$.
	Applying the monotone convergence theorem with measure $\mathbb{P}$ to the
	non-increasing sequence of non-positive
	functions $\nu \mapsto f_\nu^{-1}(t_n) - f_\nu^{-1}(\frac{1+t}{2})$,
	we obtain $\lim_{n \rightarrow \infty} g(t_n) - g(\frac{1 + t}{2}) = g(t) - g(\frac{1+t}{2})$,
	which shows that $g$ is right-continuous at $t$ since
	the decreasing sequence $\{t_n\}_{n \ge 1}$ is arbitrarily chosen
	and $g$ is non-decreasing on $(0, 1)$.
	By a similar application of the monotone convergence theorem,
	we see that $g(1) = \lim_{t \uparrow 1} g(t)$.
	Hence, our claim on $g$ is proven.

	Since $g|_{(0, 1)}$ is non-decreasing and right-continuous,
	there exists a unique measure
	$\mu_\mathbb{P} \in \mathcal{W}_2(\mathbb{R})$ such that $f_{\mu_\mathbb{P}}^{-1} = g$.
	It remains to show that ${\mu_\mathbb{P}}$ is the unique barycenter of $\mathbb{P}$.
	For any measure $\eta \in \mathcal{W}_2(\mathbb{R})$,
	by \Cref{thm:optimal_transport_on_real_line} and Fubini's theorem, we have
	\begin{align}
		\nonumber \int_{\mathcal{W}_2(\mathbb{R})} d_W(\eta, \nu)^2 \diff \mathbb{P}(\nu)
		 & = \int_0^1 \int_{\mathcal{W}_2(\mathbb{R})}
		[f_\eta^{-1}(t) - f_\nu^{-1}(t)]^2 \diff \mathbb{P}(\nu) \diff t \\
		\label{equa:rewrite_Wasserstein_barycenter_real_line}
		 & =  \int_0^1 \left[
			f_\eta^{-1}(t) - \int_{\mathcal{W}_2(\mathbb{R})} f_\nu^{-1}(t) \diff \mathbb{P}(\nu)
			\right]^2 \diff t + I(\mathbb{P}),
	\end{align}
	where the abbreviated term $I(\mathbb{P})$ is independent of $\eta$:
	\[
		I(\mathbb{P}):  = \int_0^1 \int_{\mathcal{W}_2(\mathbb{R})} [ f_\nu^{-1}(t) ]^2 \diff \mathbb{P}(\nu) -
		\left( \int_{\mathcal{W}_2(\mathbb{R})}  f_\nu^{-1}(t) \diff \mathbb{P}(\nu) \right)^2
		\diff t.
	\]
	It follows from (\ref{equa:rewrite_Wasserstein_barycenter_real_line})
	that the infimum $\inf_{\eta \in \mathcal{W}_2(\mathbb{R})}
		\int_{\mathcal{W}_2(\mathbb{R})} d_W(\eta, \nu)^2 \diff \mathbb{P}(\nu)$
	is reached by $\eta$ if and only if
	$f_\eta^{-1}(t) = \int_{\mathcal{W}_2(\mathbb{R})} f_\nu^{-1}(t) \diff \mathbb{P}(\nu) = g(t) = f_{\mu_\mathbb{P}}^{-1}(t)$
	for almost every $t \in [0, 1]$.
	Therefore, ${\mu_\mathbb{P}}$ is a barycenter of $\mathbb{P}$
	and its uniqueness follows from
	the injectivity of the embedding in  \Cref{thm:optimal_transport_on_real_line}.
\end{proof}

Let us also mention the following result which is a particular case of the main result in \cite{ma2023absolute}:

\begin{thm}\label{thm-ac-was-R}
Let $\mathbb{P} \in \mathcal{W}_2(\mathcal{W}_2(\mathbb{R}))$ be a probability measure on $(\mathcal{W}_2(\mathbb{R}), d_W)$. Assume that $\mathbb{P}(\was^{ac}(\R))>0$ where $\was^{ac}(\R)$ is the set of probability measures in $\was(\R)$ that are absolutely continuous with respect to the Lebesgue measure on $\R$.
Then the unique barycenter $\mu_{\mathbb{P}}$ also belongs to $\was^{ac}(\R)$.
\end{thm}

%%%%%%%%%%%%%%%%%%%%%%%%%%%%%%%%%%%%%%%%%%%%
\section{Fine properties of Optimal Transport over a metric graph}\label{sec-fine}

\subsection{Restriction property}\label{section-restriction-pro}

%%%%%%%%%%%%%%%%%%%%%%%%%%%%%%%%%%%%%%%%%%%%%%%%%%%

It is a classical result that the restriction $\tilde{\pi}$ of an optimal plan $\pi$ (namely $0\not\equiv \tilde{\pi} \leq \pi$) is still  optimal, see for instance \cite[Theorem 4.6]{villani2009optimal} for a proof. This section aims at generalising this property to barycenters of measures, see Proposition \ref{prop-restri-barycenters}. While we shall make use of this result in the case of a metric graph, the natural set-up for this property is that of Polish spaces. We believe this result can be useful in other contexts so we state and prove it at this level of generality.

\medskip

Throughout this section, $(X,d)$ is a \emph{Polish space}. Let us recall the existence of a \emph{measurable} mapping (see Theorem \ref{thm-sel-opt-plan})
\begin{eqnarray*}
\pi : \was(X) \times \was(X)  & \longrightarrow  &   \was(X\times X), \\
						(\mu,\nu)  & \longmapsto    &  \pi_{\mu}^{\nu}
						\end{eqnarray*} 
such that $\pi_{\mu}^{\nu} \in \Gamma^o(\mu,\nu)$ for all $\mu,\nu$. 

\medskip

We fix  $\mu \in \was(X)$. Let us consider $ 0 \not\equiv \tilde{\mu}^1 \nleq \mu$ and $\mu^1:= \frac{1}{\tilde{\mu}^1(X)} \tilde{\mu}^1$ its associated probability measure. Observe that

\begin{itemize}

\item The measure $\tilde{\mu}^1$ induced a decomposition of $\mu$  into a convex combination of probability measures ($\lambda \in (0,1)$):
\begin{equation}\label{eq-decompo-mu}
 \mu = \lambda \mu^1 + (1- \lambda) \mu^2.              
\end{equation}

\item There exist  nonnegative and bounded measurable functions $g_1$ and $g_2$ such that
$$ \mu^i = g_i \,\mu,$$
for $i\in \{1,2\}$.

\item For any $\nu \in \was(X)$, the map $g_i$ induces an optimal plan

\begin{equation*}\label{eq-optimal-res-plans}
   g_i \pi_{\mu}^{\nu}  \in \Gamma^o(\mu^i,\nu^i),
\end{equation*}   
where  $ \nu^i  = {p_2}_{\sharp} (g_i \pi_{\mu}^{\nu}).$
\end{itemize}

\begin{rmk}
Note that $\lambda g_1+(1-\lambda)g_2 =1 $ $\mu$-a.e. The last item is a particular instance of the inheritance of optimality through the restriction of optimal plans as recalled at the beginning of this section.

\end{rmk}

In order to state the restriction property for barycenters of measures, we first need to define

\begin{defn}[Restriction mappings]\label{def-restrict-map}

For each $ i \in \{1,2\}$, we define the \emph{restriction mapping} associated to \eqref{eq-decompo-mu} by
\begin{eqnarray*}
\mathcal{R}^i : \was(X) & \longrightarrow & \was(X) \\
				\nu  & \longrightarrow & \nu^i
\end{eqnarray*}
where $ \nu^i  = {p_2}_{\sharp} (g_i \pi_{\mu}^{\nu}).$
\end{defn}

Let us prove

\begin{lem} The restriction mappings $\mathcal{R}^i$, $i \in \{1,2\}$, are measurable.

\end{lem}

\begin{proof}
Given $\pi_1, \pi_2 \in \mathcal{W}_2(X \times X)$, if $\sigma \in \mathcal{W}_2(X^2 \times X^2)$
	is an optimal transport plan between $\pi_1$ and $\pi_2$, then
	\begin{eqnarray}\label{eq-X4-plan-esti}
		d_W(\pi_1, \pi_2)^2 &= &\int_{X^2 \times X^2} d(x_1, x_2)^2 + d(y_1, y_2)^2 \diff \sigma((x_1, y_1), (x_2, y_2)) \nonumber\\
							&\ge & \int_{X \times X} d(y_1, y_2)^2 \diff\, [p_2 \times p_2]_{\#} \sigma\, (y_1, y_2)
		\ge d_W({p_2}_{\#} \pi_1, {p_2}_{\#} \pi_2)^2,
	\end{eqnarray}
	which implies the continuity of the push-forward map ${p_2}_{\#}$.
	
	Now recall that $\nu^i ={p_2}_{\sharp} (g_i \pi_{\mu}^{\nu})$. Therefore, we are done if we can prove the  Borel measurability of $\pi \longmapsto g_i\pi$ for $\pi \in \mathcal{W}_2(X\times X)$ such that ${p_1}_{\#} \pi =\mu$. Let us first assume that $g_i \in C^b(X)$. We shall use the characterisation of the Wasserstein convergence on $\mathcal{W}_2(X\times X)$ as a weak convergence for continuous functions $\phi :X\times X \longrightarrow \R$ such that
\begin{equation}\label{eq:conv-Wasserst}
\exists \, C>0, \exists (x_0,y_0) \in X^2; \;\forall (x,y) \in X^2 \hspace{0.5cm}  |\phi(x,y)| \leq C(1+d^2(x_0,x)+d^2(y_0,y)),   
\end{equation}	
we refer to \cite{villani2009optimal} for a proof of this characterisation.

Given a sequence $(\pi_n)_{n \in \mathbb{N}}$ such that $\pi_n \longrightarrow \pi_{\infty}$ in $\mathcal{W}_2(X\times X)$, and $\phi$ a continuous function as in \eqref{eq:conv-Wasserst}, noting that 
$$|\phi(x,y) g_i(x)| \leq \|g\|_{\infty}C(1+d^2(x_0,x)+d^2(y_0,y)),$$
we infer
$$ \int_{X\times X} \phi \, g_i \diff \pi_n \longrightarrow  \int_{X\times X} \phi \, g_i \diff \pi$$
when $n \rightarrow \infty$. This proves the continuity of the mapping $\pi \longmapsto g_i \pi$ when $g_i \in C^b(X)$.

To prove the general case, namely when $g_i\geq 0$ belongs to $L^{\infty}(\mu)$, we make use of Lusin's theorem as proved in \cite[Theorem 7.1.13]{bogachev2007measure}. Using that $(X,d)$ is a Polish space, for any $\varepsilon>0$, we get from this theorem the existence of a continuous function $f:X \longrightarrow [0,\infty)$ such that $\|f\|_{\infty} \leq \|g_i\|_{\infty} $ and  $\mu(\{x \in X; f(x) \neq g_i(x)\}) <\varepsilon$. For each positive integer $n$, let us set $f_n$ a continuous function satisfying Lusin's statement for $\varepsilon=1/n^2$. 

Since a pointwise limit of measurable maps is measurable, we are left with proving that for any $\pi \in \mathcal{W}_2(X\times X )$ such that ${p_1}_{\#} \pi =\mu$, $f_n \pi \longrightarrow g_i\pi$ in $\mathcal{W}_2(X\times X)$ when $n\rightarrow \infty$.

To this aim, let $\phi$ be a continuous function as in \eqref{eq:conv-Wasserst}; observe that
\begin{eqnarray*}
\left| \int_{X \times X} \phi \, f_n \diff \pi -\int_{X\times X} \phi \, g \diff \pi \right| & \leq & \int_{X \times X} |\phi|\, |f_n -g| \diff \pi \\
																						& \leq & C \int_{X \times X}  (1+ d^2(x_0,x) + d^2(y_0,y)) \,|f_n(x) -g_i(x)| \diff\pi.
\end{eqnarray*}

By construction and the Borel-Cantelli lemma,
$$\lim_{n \rightarrow \infty} (1+ d^2(x_0,x) + d^2(y_0,y)) |f_n(x) -g_i(x)| =0$$
for $\mu$-a.e $x$ and every $y$ thus $\pi$-almost everywhere. For $\pi \in \mathcal{W}_2(X\times X)$, the function 
$$(x,y) \longmapsto (1+ d^2(x_0,x) + d^2(y_0,y))$$
belongs to $L^1(\pi)$ the Lebesgue dominated convergence theorem yields
$$ \lim_{n \rightarrow \infty} \left| \int_{X \times X} \phi \,f_n \diff \pi -\int_{X\times X} \phi \, g \diff \pi \right| =0$$
and the proof is complete.
\end{proof}

The following proposition details our method to construct new Wasserstein barycenters by restricting an existing one.

\begin{prop}[Restriction property of Wasserstein barycenters]\label{prop-restri-barycenters}
	%\label{lem:division_of_Wasserstein_barycenter}
	Let $(X,d)$ be a proper Polish space	and $\mu \in \mathcal{W}_2(X)$ be a fixed probability measure on $X$.  %{\color{red} Proper?}
	
	We consider a decomposition as in \eqref{eq-decompo-mu}:
	$$ \mu = \lambda \, \mu^1 + (1 - \lambda) \mu^2$$
    and 	$\mathbb{P} \in \mathcal{W}_2(\mathcal{W}_2(X))$.
    
    Then, for $i\in \{1,2\}$, if $\mu$ is a barycenter of $\mathbb{P}$ then
	$\mu^i$ is a barycenter of $\mathbb{Q}^i: = {\mathcal{R}^i}_{\#}\mathbb{P}$, where $\mathcal{R}^i$ is the corresponding restriction mapping.    
    
\end{prop}

\begin{proof}
	
	Assume that $\mu = \mu_\mathbb{P}$ is a barycenter of $\mathbb{P}$ and observe that
	\begin{align*}
		    % & \min_{\eta \in \mathcal{W}_2(X)}
		  &\int_{\mathcal{W}_2(X)} d_W(\mu_\mathbb{P}, \nu)^2 \diff \mathbb{P}(\nu)
		= \int_{\mathcal{W}_2(X)} \int_{X \times X}
		d(x, y)^2 \diff \pi_{\mu}^{\nu}(x, y) \diff \mathbb{P}(\nu)                 \\
		=   & \int_{\mathcal{W}_2(X)} \int_{X \times X} d(x, y)^2
		\left[ \lambda \diff \pi_{\mu^1}^{\nu^1}(x, y) + (1- \lambda) \diff \pi_{\mu^2}^{\nu^2}(x, y) \right]
		\diff \mathbb{P}(\nu)                                               \\
		=   & \int_{\mathcal{W}_2(X)} \left[\lambda \,d_W(\mu^1, \nu^1)^2 +
		(1 - \lambda) d_W(\mu^2, \nu^2)^2\right] \diff \mathbb{P}(\nu)      \\
		\ge & \lambda \min_{\eta \in \mathcal{W}_2(X)}
		\int_{\mathcal{W}_2(X)} d_W(\eta, \nu)^2 \diff \mathbb{Q}^1(\nu) +
		(1- \lambda) \min_{\eta \in \mathcal{W}_2(X)}
		\int_{\mathcal{W}_2(X)} d_W(\eta, \nu)^2 \diff \mathbb{Q}^2(\nu),
	\end{align*}
	where we used the fact that $g_i \pi_{\mu}^{\nu}$ is an optimal transport plan
	between $\mu^i$ and $\nu^i$.
	It follows that $\mathbb{Q}^i \in \mathcal{W}_2(\mathcal{W}_2(X))$ for $i = 1, 2$.
	We claim that the last inequality above must be an equality.
	Consider the measure $\widebar \mu :=
	\lambda \, \mu_{\mathbb{Q}^1} + (1- \lambda) \mu_{\mathbb{Q}^2}$ 
	with $\mu_{\mathbb{Q}^i}$ being a Wasserstein barycenter of $\mathbb{Q}^i$.
	According to the decomposition $\nu = \lambda \, \nu^1 + (1 - \lambda) \nu^2$,
	if $\theta_i$ ($i = 1, 2$) is a transport plan between $\mu_{\mathbb{Q}^i}$ and $\nu^i$,
	then $\lambda\, \theta_1 + (1 - \lambda) \theta_2$
	is a transport plan between $\widebar \mu$ and $\nu$.
	Hence, it follows from the definitions of $\mu_{\mathbb{Q}^i}$,
	$\mathbb{Q}^i$, $\nu^i$, and $d_W(\widebar \mu, \nu)^2$ that
	\begin{align*}
		    & \lambda \min_{\eta \in \mathcal{W}_2(X)}
		\int_{\mathcal{W}_2(X)} d_W(\eta, \nu)^2 \diff \mathbb{Q}^1(\nu) +
		(1- \lambda) \min_{\eta \in \mathcal{W}_2(X)}
		\int_{\mathcal{W}_2(X)} d_W(\eta, \nu)^2 \diff \mathbb{Q}^2(\nu)                               \\
		=   & \lambda \,\int_{\mathcal{W}_2(X)} d_W(\mu_{\mathbb{Q}^1}, \nu)^2 \diff \mathbb{Q}^1(\nu) +
		(1- \lambda) \int_{\mathcal{W}_2(X)} d_W(\mu_{\mathbb{Q}^2}, \nu)^2 \diff \mathbb{Q}^2(\nu)      \\
		=   & \int_{\mathcal{W}_2(X)} \left[\lambda \,d_W(\mu_{\mathbb{Q}^1}, \nu^1)^2 +
		(1 - \lambda) d_W(\mu_{\mathbb{Q}^2}, \nu^2)^2\right] \diff \mathbb{P}(\nu)                      \\
		\ge & \int _{\mathcal{W}_2(X)} d_W(\widebar \mu, \nu)^2 \diff \mathbb{P}(\nu).
	\end{align*}
	Hence, if our claim were false,
	the expression $\int_{\mathcal{W}_2(X)} d_W(\cdot, \nu)^2 \diff \mathbb{P}(\nu)$
	would admit a strictly smaller value on
	the measure $\widebar \mu$ than on the barycenter measure $\mu$. The proof is complete.
\end{proof}

A key property used frequently in conjunction with 
\Cref{prop-restri-barycenters} is that
the map $\mathcal{R}^i$ sends a probability measure $\nu$ to the measure $\nu^i$ that is absolutely
continuous with respect to $\nu$.
The following corollary presents two direct consequences of this property.

\begin{coro}
	\label{coro:absolute_continuity_and_restrictions_of_barycenters}
	Let $(X, d)$ be a proper metric space and $\mathbb{P} \in \mathcal{W}_2(\mathcal{W}_2(X))$
	with barycenter $\mu_\mathbb{P} \in \mathcal{W}_2(X)$. Assume the existence of a decomposition as in \eqref{eq-decompo-mu}:
	$$\mu_\mathbb{P} = \lambda \, \mu^1 + (1 - \lambda) \mu^2$$
	with $\mu^i \in \mathcal{W}_2(X)$ for $i = 1, 2$ and $\lambda \in (0, 1)$. Keeping the notation in Proposition \ref{prop-restri-barycenters}, for any measure $\eta$ on $X$,
	the following properties hold
	\begin{enumerate}
		\item if the measure $\mathbb{P}$ gives mass to the set of probability measures
			that are absolutely continuous with respect to $\eta$,
			then so do the measures $\mathbb{Q}^1$ and $\mathbb{Q}^2$;
		\item if the measure $\mathbb{P}$ assigns full mass to the set of probability measures
			that are absolutely continuous with respect to $\eta$,
			then so do the measures $\mathbb{Q}^1$ and $\mathbb{Q}^2$.
	\end{enumerate}
\end{coro}

\begin{proof}
	\Cref{prop-restri-barycenters} provides us with
	two continuous maps $\mathcal{R}^1, \mathcal{R}^2: \mathcal{W}_2(X) \rightarrow \mathcal{W}_2(X)$
	such that 
	$$\lambda\,\mathcal{R}^1(\nu) + (1 - \lambda) \mathcal{R}^2(\nu) = \nu$$
	 for $\nu \in \mathcal{W}_2(X)$ and $\mu^i$ ($i = 1, 2$) is a barycenter
	of  $\mathbb{Q}^i: = {\mathcal{R}^i}_{\#} \mathbb{P}$.
	Observe that the probability measures $\mathcal{R}^1(\nu), \mathcal{R}^2(\nu)$
	are absolutely continuous with respect to $\nu$.
	Hence, given a measure $\eta$ on $X$,
	if $\nu$ is absolutely continuous with respect to it,
	then so are the measures $\mathcal{R}^1(\nu)$ and $\mathcal{R}^2(\nu)$.
	The two assertions in the corollary then follow
	directly. 
\end{proof}

\subsection{Subsets of geodesics in a metric graph and induced properties}

In this part, we fix  $\de$ a \emph{minimising} oriented edge of $(G,d)$. The result we obtain here will be used in the next part.
\smallskip

 We decompose the space of geodesics in $(G,d)$ starting at a point lying in $e$ into three pieces. This decomposition of the space of geodesics depends on the orientation of $e$. We set

\begin{equation}
\geoE  = \geo (\de) \sqcup \geo(G)^{\de,+} \sqcup  \geo(G)^{\de, -},
\end{equation}

where

\begin{itemize}

\item $\geoE =\{\gamma \in \geo (G); \gamma(0) \in e\}$

\item $ \geo (\de)=\{\gamma \in \geoE; \gamma([0,1]) \subset \de\}$

\item $\geo(G)^{\de,+} = \{\gamma \in \geoE  \setminus  \geo (\de); \de_1 \in \ima (\gamma) \}$

\item $\geo(G)^{\de,-} = \{\gamma \in \geoE  \setminus  \geo (\de); \de_0 \in \ima (\gamma) \}$
\end{itemize}

Note that since $\de$ is minimising, the  set  $\geo(G)^{\de,+} \cap \geo(G)^{\de,-}$ has to be empty.

\begin{lem} The sets $\geo (\de),\geo(G)^{\de,+}, \geo(G)^{\de,-}$ are Borel subsets of $\geo(G)$ equipped with the topology of uniform convergence.

\end{lem}

\begin{proof}
Recall that $\geo(G) \subset Lip([0,1], G)$. It is clear that $ \geo (\de)$ is a closed subset of $\geo(G)$, hence measurable,  since $\de$ is closed. Let us now treat the case of 
$\geo(G)^{\de,+}$, the case $\geo(G)^{\de,-}$ can be proved by similar arguments. Observe that
$$\geo(G)^{\de,+} =\{ \gamma \in \geoE; \min_{t \in [0,1]} d(\de_1,\gamma(t))=0 \mbox{ and } \max_{t \in [0,1]} d(e,\gamma(t)) >0\}$$

where we use the fact that $d(\de_0,\de_1) = l(e)$ to get this equality. Consequently

$$ \geo(G)^{\de,+} = \cup_{n \in \mathbb{N}\setminus\{0\}} \mathcal{G}_n,$$

where $ \mathcal{G}_n = \{\gamma \in \geoE; \min_{t \in [0,1]} d(\de_1,\gamma(t))=0 \mbox{ and } \max_{t \in [0,1]} d(e,\gamma(t)) \geq 1/n\}$. Using the compactness of $[0,1]$, one easily obtain that each set $\mathcal{G}_n$ is a closed subset of $\geo(G)$ hence $\geo(G)^{\de,+}$ is measurable.

\end{proof}

As a consequence of the existence of this measurable partition of $\geoE$, we get a decomposition of any measure $\Pi$ in $\P(\geo(G))$ such that $\gamma(0) \in \de$ $\Pi$-almost surely. We simply consider the restriction of $\Pi$ to each of the three subsets $\geo (\de),\geo(G)^{\de,+},$ and $ \geo(G)^{\de,-}$ above.

\begin{defn}[Decomposition of a dynamical plan]\label{def-decompo-dyn-plan}
Let $\Pi \in \P (\geo(G))$ and $\de \in \oEdge$. Assume that $\gamma(0) \in \de$ for $\Pi$-a.e $\gamma$. Then, $\Pi$ has a unique decomposition in $\P(\geo (G))$.

\begin{equation}\label{eq-decompo-dyn-plan}
\Pi  = \Pi^{\de}  + \Pi^+  + \Pi^{-}.
\end{equation} 

This decomposition depends on $\de$.
\end{defn}

We now introduced a collection of three Lipschitz maps, each of which will be used in conjonction with one of the above subsets of $\geo(G)$.

\begin{defn}\label{def-theMapHs}

We define the following maps

$$ \begin{array}{rcl}
  \bullet \hspace{0.2cm } h^{\de} : G &\longrightarrow &\R \\
             x & \longmapsto  & d(\de_0, x) 
\vspace*{0.2cm} \\

  \bullet \hspace{0.2cm } h^{+} : G &\longrightarrow &\R \\
             y & \longmapsto  & l(e) + d(\de_1, y)  \vspace*{0.2cm} \\

 \bullet \hspace{0.2cm }  h^{-} : G &\longrightarrow &\R \\
             y & \longmapsto  & -d(\de_0, y)  
\end{array}$$

\end{defn} 

These maps are $1$-Lipschitz maps hence Borel measurable.

\smallskip 

We conclude this part with the definition of a measurable map from $\was(G)$ to $\was(\R)$ whose properties will be proved in the next section.

\begin{defn}\label{def-theMapPhi} For a given $\bar{\mu} \in \was(e)$, and based on  the decomposition $\Pi  = \Pi^{\de}  + \Pi^+  + \Pi^{-}$ from \eqref{eq-decompo-dyn-plan}, we define a map
$$ \Phi : \was(G)  \longrightarrow  \was(\R)$$

by the formula

\begin{equation}\label{eq-Phi-defn}
\Phi(\nu):= h^{\de}_{\sharp} \nu^{\de} + h^{+}_{\sharp} \nu^{+} + h^{-}_{\sharp} \nu^{-},
\end{equation}
where for $i \in \{\de, +,-\}$, $\nu^i = {p_2}_{\sharp} (\Pi_{\bar{\mu}}^{\nu})^i$.
\end{defn}

\begin{rmk}\label{rmk-Rmk-on-Phi}
Note the alternative characterisation of the three terms in the definition of $\Phi(\nu)$.

$$ \begin{array}{rcl}
\Phi(\nu)\res_{[0,l(e)]} & = &h^{\de}_{\sharp} \nu^{\de} \\
\Phi(\nu)\res_{(l(e),\infty)} & = & h^{+}_{\sharp} \nu^{+}  \\
\Phi(\nu)\res_{(-\infty,0)} & = & h^{-}_{\sharp} \nu^{-}
\end{array}
$$

Besides, with our convention, for $\de \sim [0,l(e)]$, note that $h^{\de}(x)= d(\de_0,x)=x$ whenever $x \in\de$. Therefore, given any $\mu \in \was(e)$, $h^{\de}_{\sharp} \mu = \mu$ and the map $\Phi$ (restricted to $\was(e)$) is \emph{surjective onto} $\was([0,l(e)])$.
\end{rmk}

\section{Branched covering {\&} quasi regularity}

In this part we study the properties of the map $\Phi$. We call $\Phi$ a ``branched covering" because of \Cref{lem-const-preimage} stating that the cardinals of the preimages of the maps involved in the definition of $\Phi$ are locally constant.

Let us recall that the edge $\de$ involved in the definition of $\Phi$ is assumed to be minimising.

\subsection{A branched covering map from $\was(G)$ to $\was(\R)$ }

The goal of this part is to prove 

\begin{thm}\label{thm-prop-branched-cover}
The map $\Phi : \was(G)  \longrightarrow  \was(\R)$ satisfies the following properties:

\begin{enumerate}

\item The map $\Phi$ preserves the regularity with respect to the reference measure:
$$ \nu \ll \H \Leftrightarrow \Phi(\nu) \ll \Leb^1,$$
where $\Leb^1$ is the Lebesgue measure on $\R$.
\item $\forall \mu \in \was(e), \, \forall \nu \in \was(G)$,
$$d_W (\Phi(\mu), \Phi(\nu)) \geq d_W(\mu,\nu),$$

\item When $\mu=\bar{\mu}$ (where $\bar{\mu}$ is as in \Cref{def-theMapPhi}), equality holds:
$$ d_W (\Phi(\bar{\mu}), \Phi(\nu)) = d_W(\bar{\mu},\nu).$$
\end{enumerate}

\end{thm}

Before proving this theorem, we need to establish some intermediate results.

\begin{lem}\label{lem-const-preimage} For $i \in \{\de, +,-\}$, there exists a discrete set $D^i$ such that the cardinal of $(h^i)^{-1}(\tilde{x})$ is constant on each connected component of $\R\setminus D^i$.

\end{lem}

\begin{proof}
Let us fix $i \in \{\de, +,-\}$ and recall that the interior of an edge $\de$ is unambiguously identified with the interval $]0,l(e)[$. Define $f_i$ to be $\de_0$ if $i \in \{ \de,-\}$ and $f_i=\de_1$ otherwise. Then, consider $D^i= h^i(V \cup \overrightarrow{ Cut}_{f_i})$. According to Lemma \ref{lem:countability-cut}, the set $V \cup \, \overrightarrow{ Cut}_{f_i}$ is a discrete set. We infer the same property for $D^i$ since  $ (h^i)^{-1}([-R,R]) \subset  \widebar{B} (f_i, R)$ which implies that $(h^i)^{-1}([-R,R]) \cap (V \cup \overrightarrow{ Cut}_{f_i})$ is finite. Now, for $x \in (h^i)^{-1}(\tilde{x})$ with $\tilde{x} \not\in D^i$, we claim that $h^i$ is differentiable at $x$ and $|\nabla h_i(x)|=1$. The differentiability of $h^i$ follows from the fact that $x \not\in V \cup \overrightarrow{ Cut}_{f_i}$ which implies that the distance function $d(f_i, \cdot)$ is differentiable at $x$, we get the result since $h^i$ coincides up to an additive constant and switching the sign with $d(f_i, \cdot)$. The local inverse theorem then gives us the fact that $\sharp ((h^i)^{-1}(\tilde{x})$ is locally constant and the proof is complete.
\end{proof}

\begin{rmk}
As a consequence of the proof, note that for any interval $I_x \subset G \setminus (V \cup \overrightarrow{ Cut}_{f_i})$ the inverse of $h^i\res_{I_x}$ is $1$-Lipschitz as well hence measurable.
\end{rmk}

We also need the following result which is a consequence of the  disintegration theorem as stated in \cite{ambrosio2008gradient, Santambrogio2015} for instance.

\begin{lem}\label{lem-disinte-branc-cov} Let $(G,d)$ be a metric graph and $h: G \longrightarrow \R$ be a Lipschitz map. Assume the existence of an (open) interval $E \subset \R$ and $S= \sqcup_{i=1}^k s_i \subset G$ such that for each $i \in \{1,\cdots, k\}$, $h|_{s_i} : s_i \longrightarrow E$ is a bi-Lipschitz homeomorphism onto $E$ and $h^{-1}(E) =S$. For $\nu \in \was(G)$, set $\tilde{\nu}:= h_{\sharp} \nu$. 

Then, $\tilde{\nu}\res_E= h_{\sharp} (\nu\res_S)$ and

\begin{equation}\label{eq-disinte-branc-cov} \nu\res_S = \int_E p_{\tilde y} \, \diff \tilde{\nu} (\tilde{y})
\end{equation}
where for $\tilde{\nu}\res_E$-a.e. $\tilde{y}$, $p_{\tilde y}$ is a  probability measure supported in the finite set $h^{-1}(\tilde{y})$.

Moreover, if $\tilde{\theta} \in \mathcal{M}(\R^2)$ is a Borel measure on $\R^2$ such that its second marginal satisfies $(p_2)_{\sharp} \tilde{\theta}\ll \tilde{\nu}\res_E$ then there exists a plan $\theta \in \was (\R \times G)$ such that 
$$ (Id,h)_{\sharp} \theta = \tilde{\theta}.$$

\end{lem}

\begin{proof}

The proof of $\tilde{\nu}\res_E= h_{\sharp} (\nu\res_S)$ follows from $h^{-1}(E) =S$. The formula \eqref{eq-disinte-branc-cov}, is then a straightforward application of the disintegration theorem. The fact that $p_{\tilde{y}}$ is concentrated on a finite set follows from the equality $S= \sqcup_{i=1}^k s_i$.

Let us now prove the second part of the lemma. Set $f$ the density function defined by
$$ (p_2)_{\sharp} \tilde{\theta}= f \, \tilde{\nu}\res_E .$$
Observe that 
\begin{equation}\label{eq-decompo-tilde-nu}
\tilde{\nu}\res_E = \sum_{i=1}^k \tilde{\nu}^i
\end{equation} 
where $\tilde{\nu}^i = h_{\sharp} (\nu\res_{s_i})$ and 
$$ (h|_{s_i})^{-1}_{\sharp}\tilde{\nu}^i =(h|_{s_i})^{-1}_{\sharp} (h_{\sharp} (\nu\res_{s_i})) = \nu\res_{s_i}.$$
Disintegrate $\tilde{\theta}$ with respect to its second marginal:
$$ \tilde{\theta} =\int  \tilde{\theta}_{\tilde{y}}\; f(\tilde{y}) \diff \tilde{\nu}\res_E(\tilde{y}).$$
Then decompose $\tilde{\theta}$ as follows
$$\tilde{\theta} := \sum_{i=1}^k \tilde{\theta}^i$$
where 
$$ \tilde{\theta}^i =\int  \tilde{\theta}_{\tilde{y}}\; f(\tilde{y}) \diff \tilde{\nu}^i(\tilde{y}).$$

Note that by assumption on the second marginal of $\tilde{\theta}$, for any $i \in \{1,\cdots,k\}$,
\begin{equation}\label{eq-supp-thetai}
 \tilde{\theta}^i (\R \times (\R\setminus E)) =0.
\end{equation}

We now define the measure $\theta$ as follows
$$ \theta =  \sum_{i=1}^k {\theta}^i$$
where 
$$ {\theta}^i := (Id, (h|_{s_i}) ^{-1})_{\sharp} \tilde{\theta}^i$$
is well-defined thanks to \eqref{eq-supp-thetai}.

Now, observe that
\begin{eqnarray*}
(Id,h)_{\sharp} \theta   & = &  \sum_{i=1}^k (Id,h)_{\sharp} {\theta}^i  \\
						& = &   \sum_{i=1}^k (Id,h|_{s_i})_{\sharp} {\theta}^i  \\
						& = & \sum_{i=1}^k (Id,h|_{s_i})_{\sharp} ((Id, (h|_{s_i}) ^{-1})_{\sharp} \tilde{\theta}^i)\\
						& = & \sum_{i=1}^k \tilde{\theta}^i\\
						&=&   \tilde{\theta}.
\end{eqnarray*}
\end{proof}

We are now in position to tackle the proof of Theorem \ref{thm-prop-branched-cover}.

\begin{proof}(of Theorem \ref{thm-prop-branched-cover}).
\begin{enumerate}

\item By definition of $\Phi$, it suffices to prove that $\nu^i \ll \H$ iff $h^{i}_{\sharp} \nu^{i}\ll \Leb^1$ for any $i \in \{\de, +,-\}$. Since each $h^i$ is a Lipschitz map, the backward implication is a standard result (recall that $\H$ is the $1$-Hausdorff measure induced by the graph distance). To prove the reverse implication, we make use of Lemma \ref{lem-const-preimage}. Let $I \subset \R \setminus D^i$ be a connected component of $\R \setminus D^i$. Then $(h^i)^{-1} (I) = \sum_{j=1}^k s_j$ where each $s_j$ is an interval within an edge $e_j$, and $e_l$ and $e_j$ are distinct whenever $l\neq j$. We denote by $(h_j^i)^{-1}$ the inverse map of the bijection $h^i: s_j \longrightarrow I$. As explained in the proof of Lemma \ref{lem-const-preimage}, the map $(h_j^i)^{-1}$ is 1-Lipschitz hence measurable. By definition,
$$ ((h^i)_{\sharp} \nu^i)\res_{I}  = \sum_{j=1}^k (h^i)_{\sharp} (\nu^i\res_{s_j}).$$
Moreover, $h^i|_{s_j}$ being a biLipschitz homeomorphism, $(h^i)_{\sharp} (\nu^i\res_{s_j})$ is a.c. with respect to $\Leb^1$ if (and only if) $\nu^i\res_{s_j}$ is a.c. with respect to $\H$ hence the result.

\item Let  $\mu \in \was(e)$, and  $\nu \in \was(G)$. Given $\tilde{\theta} \in \Gamma^o (\Phi(\mu),\Phi(\nu))$, our goal is to construct a plan $\theta \in \Gamma(\mu,\nu)$ such that
$$ \int d^2(x,y) \, \diff \theta \leq \int |\tilde{x}-\tilde{y}|^2 \, d\tilde{\theta}.$$

Pursuing the approach initiated to prove the first item, we first construct a partition of the real line as follows. Let us set 
$$D:= D^+ \cap (l(e),\infty) \sqcup D^- \cap (-\infty, 0) \sqcup D^{\de} \cap [0,l(e)],$$
where the sets $D^i$, $i \in\{+,-,\de\}$ are discrete subsets of $\R$ defined in Lemma \ref{lem-const-preimage}, in particular the set $D$ is discrete as well.

Consequently, we get an at most countable partition of the real line $\{E_b, b \in N\}$ where either $E_b=\{d\}$ for some $d\in D$ or $E_b$ is a connected component of $\R\setminus D$. We denote by $A \subset N$ the subset for which $\{E_b, b\in A\}$ is a partition of $[0,l(e)]$.

The next step consists in decomposing the transport plan $\tilde{\theta}$ by making use of the partition above. 
$$\tilde{\theta} = \sum_{a \in A, b \in N} \tilde{\theta}\res_{E_a \times E_b}.$$

Now for a fixed $(a,b) \in A\times N$, let us set $\Xi_{a,b}:=\tilde{\theta}\res_{E_a \times E_b}$ and assume $\Xi_{a,b} \not\equiv 0$. Note that $E_a \subset [0,l(e)]$ and by assumption $d_{\de_0} : \de \longrightarrow [0,l(e)]$, where $d_{\de_0}(x)= d (\de_0,x)$, is an isometry. Thus we only need to focus on the second marginal of $\Xi_{a,b}$. In the case when $E_b=\{d\}$ with $d \in D_i$ then $(h^i)^{-1}(d)=\{g_1,\cdots,g_s\}$ ($i \in \{+,-,\de\}$) is a finite set. Thus if we set $\tilde{\nu}= h^i_{\sharp} \nu$, we have
$$ \tilde{\nu}(\{d\})= \sum_{i=1}^s  \nu(\{g_i\}).$$
Then we simply consider the transport plan 
$$\theta_{a,b}:= \sum_{i=1}^s \nu(\{g_i\}) \big(\,{p_1}_{\sharp}\Xi_{a,b}\big) \otimes \delta_{g_i}.$$

Since by definition $\tilde{\theta}_{a,b}(\R \times (\R \setminus \{d\}))=0$, we have

\begin{equation}\label{eq-build-planI}
(Id,h^i)_{\sharp} \theta_{a,b} = \tilde{\theta}_{a,b}.
\end{equation} 

It remains to treat the case where $E_b$ is an (open) connected component of $\R\setminus D$, note that by construction $E_b$ is contained in one of these sets: $(-\infty,0), [0,l(e)], (l(e),\infty)$ and thus $\tilde{\nu}\res_{E_b}= ((h^i)_{\sharp} \nu)\res_{E_b}$ for a uniquely defined $i\in \{+,-,\de\}$. Thanks to Lemma \ref{lem-disinte-branc-cov}, there exists a plan $\theta_{a,b}$ such that 
\begin{equation}\label{eq-build-planII}
(h^i)_{\sharp} \theta_{a,b} = \Xi_{a,b}.
\end{equation}

We finally define $\theta$ as
$$\theta := \sum_{a \in A,b \in N} \theta_{a,b}.$$

By definition of $\Phi$ and the partition $\{E_b, b \in N\}$, the properties \eqref{eq-build-planI} and \eqref{eq-build-planII} give us that $\theta \in \Gamma(\mu,\nu)$.

We are left with proving the inequality $ \int d^2(x,y) \, \diff \theta \leq \int |\tilde{x}-\tilde{y}|^2 \, d\tilde{\theta}$ that gives the thesis by definition of the Wasserstein distance. As already observed above, for each $E_b$, $b \in B$, $E_b$ is contained in one of these sets: $(-\infty,0), [0,l(e)], (l(e),\infty)$. Let us treat the case where $E_b \subset (l(e),\infty)$, the others are proved similarly. By what precedes,
\begin{eqnarray*}
\int |\tilde{x} -\tilde{y}|^2 \, d\Xi_{a,b} &=&  \int |h^{\de}(x) -h^{+} (y)|^2 \, \diff \theta_{a,b} (x,y) \\
											&=& \int |d(\de_0,x) -(l(e) +d (\de_1,y))|^2 \, \diff \theta_{a,b} (x,y) \\
											& =& \int|-d(\de_1,x)  -d (\de_1,y)|^2 \, \diff \theta_{a,b} (x,y) \\
											&\geq & \int d^2(x,y) \, \diff \theta_{a,b} (x,y),
\end{eqnarray*} 
where we use the property $d(\de_0,\de_1) =l(e)$ in the second equality.

\item Let us now assume that $\mu=\bar{ \mu}$ and prove $d_W (\Phi(\bar{\mu}), \Phi(\nu)) = d_W(\bar{\mu},\nu)$. The proof is mainly based on the following observation:
\begin{equation}\label{eq-dist-prop}
\left. \text{\parbox{0.8\linewidth}{
\begin{itemize}
\item for $\Pi^{\de}$-a.e geodesic $\gamma \in \geo_{\gamma(0) \in \de}(G)$, 
$$ d(\gamma(0), \gamma(1)) \; = \; |d(\de_0, \gamma(0)) - d(\de_0, \gamma(1))|$$
\item for $\Pi^{+}$-a.e geodesic $\gamma \in \geo_{\gamma(0) \in \de}(G)$,
$$\begin{array}{rcl}
\phantom{vvvvvv} d(\gamma(0), \gamma(1)) & = &  d(\de_1, \gamma(0)) + d(\de_1, \gamma(1)) \\ & = & - d(\de_0, \gamma(0)) + l(e)  + d(\de_1, \gamma(1)), 
\end{array}$$
\item for $\Pi^{-}$-a.e geodesic $\gamma \in \geo_{\gamma(0) \in \de}(G)$,
$$ d(\gamma(0), \gamma(1)) \; = \; d(\de_0, \gamma(0)) + d(\de_0, \gamma(1)).$$
\end{itemize}
}}\right\}\end{equation}

We are then left with estimating $d_W (\Phi(\bar{\mu}), \Phi(\nu))$ from above. To this aim, we  set $\pi^i := (\cdot_0,\cdot_1)_{\sharp} \Pi^i$ for $i\in\{+,-,\de\}$ and compute, using the definition of $\Phi$,
\begin{eqnarray*}
d_W^2 (\Phi(\bar{\mu}), \Phi(\nu)) & \leq & \sum_{i \in \{+,-,\de\}} \int |h^{\de}(x)- h^i(y)|^2 \, d\pi^i(x,y) \\
						& = &  \int |d(\de_0, x) - d(\de_0, x)|^2 \, d\pi^{\de}(x,y) \\
						& \phantom{=} & \hspace{2 cm}+ \int |d(\de_0, x) - (l(e) + d(\de_1,y))|^2 \, d\pi^{+}(x,y) \\
						& \phantom{=} & \hspace{4 cm}  + \int |d(\de_0, x) - (-d(\de_0, y))|^2 \, d\pi^{-}(x,y) \\
						&= &  \sum_{i \in \{+,-,\de\}} \int d^2(x,y) \, d\pi^i(x,y)= d_W^2(\bar{\mu}, \nu), 
\end{eqnarray*} 
where the penultimate equality follows from (\ref{eq-dist-prop}).

\end{enumerate}

\end{proof}

%%%%%%%%%%%%%%%%%%%%%%%%%%%%%%%%%%%%%%%%
\subsection{Proof of the main results}\label{sec-main}

The goal of this part is to prove our main results which we recall below:

\begin{thm}[Quasi regularity of Wasserstein barycenters on metric graphs]
	\label{thm:almost_absolute_continuity}
	Let $(G,d, \mathcal{H})$ be a metric measure graph.
	Consider a measure $\mathbb{P} \in \mathcal{W}_2(\mathcal{W}_2(G))$ that
	gives mass to absolutely continuous measures on $(G,d)$.
	If $\mu_{\mathbb{P}}$ is a barycenter of $\mathbb{P}$, then the restriction of $\mu_\mathbb{P}$
	to the interior of any given edge is absolutely continuous.
	Therefore, if $\mu_\mathbb{P}$ is not absolutely continuous
	then its singular part has to be a sum of Dirac measures located at vertices of $G$.
\end{thm}
Then we prove the Wasserstein barycenter is not unique (\Cref{prop-non-unique}).

The proof of the above theorem is in two steps. In the first one, we solve the particular case when the barycenter 
is concentrated within an edge of $(G,d)$. To do so, we build upon the branched Wasserstein covering from Definition \ref{def-theMapPhi}. In the second step, we take advantage of the restriction property for Wasserstein barycenters from Section \ref{def-theMapPhi} to reduce the proof 
in the general case to that of the first step.

\subsubsection{When the barycenter is concentrated in an edge}

\begin{lem}
	\label{lem:barycenter_supported_on_an_edge}
	Let $(G,d)$ be a metric graph.
	Fix an oriented edge ${\de}$ of $(G,d)$ and
	a probability measure $\mathbb{P} \in \mathcal{W}_2(\mathcal{W}_2(G))$.
	Suppose that $\mathbb{P}$ has a barycenter $\mu_\mathbb{P} \in \mathcal{W}_2(G)$
	that is supported in the edge ${\de}$ of $(G,d)$.
	Denote by $\Phi: \mathcal{W}_2(G) \rightarrow \mathcal{W}_2(\mathbb{R})$
    the map	introduced in Definition \ref{def-theMapPhi} with $\bar{\mu}=\mu_{\mathbb{P}}$ 
    and define $\mathbb{Q}: = {\Phi}_{\sharp} \mathbb{P}$.

	Then the quantile function of $\Phi(\mu_\mathbb{P})$
	is determined by the quantile function of $\mathbb{Q}$ as follows:\\
	for $t \in [0, 1]$,
\begin{equation*}
	f_{\Phi(\mu_{\mathbb{P}})}^{-1}(t) =
		\begin{cases}
			0 \quad                          & \text{ if } f_{\mu_\mathbb{Q}}^{-1}(t) < 0        \\
			f_{\mu_\mathbb{Q}}^{-1}(t) \quad & \text{ if } 0 \le f_{\mu_\mathbb{Q}}^{-1}(t) \le l({\de}) \\
			l({\de}) \quad                   & \text { if } f_{\mu_\mathbb{Q}}^{-1}(t) > l({\de}).
		\end{cases}
	\end{equation*}
\end{lem}

\begin{proof}
	Since $\mu_\mathbb{P}$ is a barycenter of $\mathbb{P}$ that is supported in the edge ${\de}$,
	according to the second and third items of Theorem \ref{thm-prop-branched-cover} we  have
	\begin{align}\label{equa:reduction_map_push_forward_barycenter}
		\inf_{ \mu \in \mathcal{W}_2({\de}) } \nonumber
		  & \int_{ \mathcal{W}_2(G) } d_{W}(\Phi(\mu), \Phi(\nu))^2 \diff \mathbb{P}(\nu)
		\geq \inf_{ \mu \in \mathcal{W}_2({\de}) }
		\int_{\mathcal{W}_2(G) } d_{W}(\mu, \nu)^2 \diff \mathbb{P}(\nu)                                \\
		\stackrel{\mu_{\mathbb{P}} \mbox{ \small barycenter}}{=} & \int_{ \mathcal{W}_2(G) } d_{W}(\mu_{\mathbb{P}}, \nu)^2 \diff \mathbb{P}(\nu)
		\stackrel{\mbox{ Item 3.}}{=} \int_{ \mathcal{W}_2(G) }
		d_{W}(\Phi(\mu_\mathbb{P}), \Phi(\nu))^2 \diff \mathbb{P}(\nu).
	\end{align}
	
	By assumption $\mu_\mathbb{P} \in 	\mathcal{W}_2({\de}) $ thus  all the terms in (\ref{equa:reduction_map_push_forward_barycenter}) are equal.

	Now, using the surjectivity property explained in \Cref{rmk-Rmk-on-Phi}, and the definition of  $\mathbb{Q}: = {\Phi}_{\#} \mathbb{P}$,
	 we can rephrase (\ref{equa:reduction_map_push_forward_barycenter})  as
	
	\begin{align}\label{equa:barycenter_epsilon_equality_after_push_forward}
		\inf_{ \mu \in \mathcal{W}_2([0, l({\de})]) }
		\int_{\mathcal{W}_2(\mathbb{R}) } d_{W}(\mu, \nu)^2 \diff \mathbb{Q}(\nu) &= \inf_{ \mu \in \mathcal{W}_2({\de}) } 
		  \int_{ \mathcal{W}_2(G) } d_{W}(\Phi(\mu), \Phi(\nu))^2 \diff \mathbb{P}(\nu)   \nonumber \\
		&= \int_{\mathcal{W}_2(\mathbb{R}) } d_{W}(\Phi(\mu_\mathbb{P}), \nu)^2 \diff \mathbb{Q}(\nu).
	\end{align}

		%%%%%%%%%%%%%%%%%%%%%%%%% CUT %%%%%%%%%%%%%%%%%%%%%%%%%%%%%%%%%%%%%%%%%%%%
	
	Denote by $\mu_\mathbb{Q}$ the unique Wasserstein barycenter of $\mathbb{Q}\in \was(\was(\R))$. According to \Cref{lem:Wasserstein_barycenter_real_line}, the quantile function of $\mu_\mathbb{Q}$ satisfies:
	$$f_{\mu_\mathbb{Q}}^{-1} = \int_{ \mathcal{W}_2(\mathbb{R}) } f_\nu^{-1} \diff \mathbb{Q}(\nu).$$
	To further simplify (\ref{equa:barycenter_epsilon_equality_after_push_forward}),
	we build upon the isometric feature of $\nu \longmapsto f_\nu^{-1}$ to get
	\begin{align*}
		\int_{\mathcal{W}_2(\mathbb{R})} d_W(\mu, \nu)^2 \diff \mathbb{Q}(\nu)
		 & = \int_0^1 \int_{\mathcal{W}_2(\mathbb{R})}
		[f_\mu^{-1}(t) - f_\nu^{-1}(t)]^2 \diff \mathbb{Q}(\nu) \diff t \\
		 & =  \int_0^1 \left[
			f_\mu^{-1}(t) - \int_{\mathcal{W}_2(\mathbb{R})} f_\nu^{-1}(t) \diff \mathbb{Q}(\nu)
			\right]^2 \diff t + I(\mathbb{Q}) \\
		 & = 
			\int_0^1 [f_{\mu}^{-1} - f_{\mu_\mathbb{Q}}^{-1}]^2 \diff \lambda + I(\mathbb{Q}),
	\end{align*}
	where the abbreviated term $I(\mathbb{Q})$ is independent of $\mu$:
	\[
		I(\mathbb{Q}):  = \int_0^1 \int_{\mathcal{W}_2(\mathbb{R})}
		[ f_\nu^{-1}(t) ]^2 \diff \mathbb{Q}(\nu) -
		\left( \int_{\mathcal{W}_2(\mathbb{R})}  f_\nu^{-1}(t) \diff \mathbb{Q}(\nu) \right)^2
		\diff t.
	\]
	Hence, (\ref{equa:barycenter_epsilon_equality_after_push_forward}) is equivalent to
	\begin{equation}
		\label{equa:barycenter_into_min_over_quantile_functions}
		\inf_{ \mu \in \mathcal{W}_2([0, l({\de})]) }
		\int_0^1 [f_{\mu}^{-1}(t) - f_{\mu_\mathbb{Q}}^{-1}(t)]^2 \diff t =
		\int_0^1 [f_{\Phi(\mu_\mathbb{P})}^{-1}(t) - f_{\mu_\mathbb{Q}}^{-1}(t)]^2 \diff t.
	\end{equation}
	By \Cref{lem:quantile_function_range_and_measure_support}, a probability measure $\mu \in \was(\R)$ is in the space
	$\mathcal{W}_2([0, l({\de})])$ if and only
	the image of $f_{\mu}^{-1}$ is contained in the interval $[0, l({\de})]$.
	Hence, a solution $\mu \in \mathcal{W}_2([0, l({\de})])$ of the minimization problem in the left-hand side
	of (\ref{equa:barycenter_into_min_over_quantile_functions}) must satisfy the requirements
	\begin{equation*}
		f_\mu^{-1}(t) = 0 \text { if } f_{\mu_\mathbb{Q}}^{-1}(t) < 0,\quad
		f_\mu^{-1}(t) = f_{\mu_\mathbb{Q}}^{-1}(t) \text{ if }
		0 \le f_{\mu_\mathbb{Q}}^{-1}(t) \le l({\de}),\quad
		f_\mu^{-1}(t) = l({\de}) \text { if } f_{\mu_\mathbb{Q}}^{-1}(t) > l({\de}).
	\end{equation*}
	In particular, the measure $\Phi(\mu_\mathbb{P}) \in \mathcal{W}_2([0, l({\de})])$ satisfies the above requirements,
	which concludes the proof.
\end{proof}

By using \Cref{lem:barycenter_supported_on_an_edge}, we can determine the barycenter 
of $\mathbb{Q}=  \Phi_{\sharp} \mathbb{P}$ provided $\mu_{\mathbb{P}}$ is concentrated in the interior of $e$.

\begin{coro}
	\label{lem:barycenter_concentrated_in_the_interior_edge}
	Let $(G,d)$ be a metric graph.
	Fix an oriented edge ${\de}$ of $(G,d)$ and
	a probability measure $\mathbb{P} \in \mathcal{W}_2(\mathcal{W}_2(G))$.
	Suppose that $\mathbb{P}$ has a barycenter $\mu_\mathbb{P} \in \mathcal{W}_2(G)$
	that is supported in the edge ${\de}$ of $(G,d)$
	and assigns no mass to the ends of ${\de}$.
	
	Then $\Phi(\mu_\mathbb{P})$ is the unique barycenter
	of $\mathbb{Q}: = {\Phi}_{\#} \mathbb{P}$.
\end{coro}

%%%%%%%%%%%%%%%%%%%%% CUT 2 %%%%%%%%%%%%%%%%%%%%%%%%%%%%%%%%%%%%%%%

\begin{proof}

	Since $\mu_\mathbb{P}$ is assumed to give no mass to the ends of ${\de}$,
	$\Phi(\mu_\mathbb{P})$ gives no mass to the endpoints of $[0, l({\de})]$.
	Hence, \cite[Chapter 0, Lemma (4.8)]{revuz2013continuous} implies
	\begin{eqnarray*}
		0 &=& f_{\Phi(\mu_\mathbb{P})} (0) = \inf \{ t\in (0, 1) \mid f_{\Phi(\mu_\mathbb{P})}^{-1}(t) > 0 \},
		\\
		\vspace{0.4cm}
		1 &=& \lim_{s \uparrow l({\de})} f_{\Phi(\mu_\mathbb{P})} (s) =
		\lim_{s \uparrow l({\de})} \inf \{ t \in (0, 1) \mid f_{\Phi(\mu_\mathbb{P})}^{-1}(t) > s \}.
	\end{eqnarray*}
	It follows from these two equalities that for $0 < t < 1$, $f_{\Phi(\mu_\mathbb{P})}^{-1}(t) \neq 0$
	and $f_{\Phi(\mu_\mathbb{P})}^{-1}(t) \neq l({\de})$.
	According to the requirements satisfied
	by $\Phi(\mu_\mathbb{P})$ in \Cref{lem:barycenter_supported_on_an_edge},
	we must have $0 \le f_{\mu_\mathbb{Q}}^{-1}(t) \le l({\de})$ and
	$f_{\Phi(\mu_{\mathbb{P}})}^{-1}(t) = f_{\mu_\mathbb{Q}}^{-1}(t)$ for $t \in (0, 1)$.
	Hence, $\Phi(\mu_\mathbb{P}) = \mu_\mathbb{Q}$ is the unique barycenter of $\mathbb{Q}$.
\end{proof}

The proof of our main result in the case $\mu_{\mathbb{P}}(\ide)=1$ is a straightforward consequence of the corollary above as explained below. We then use the restriction property of Wassertein barycenters to prove the general case:

%%%%%%%%%%%%%%%%%%%%%%%%%%%%%%%%%%%%%%%%%%%%%%%%%%%%%%%%%
\subsubsection{Proof of Theorem \ref{thm:almost_absolute_continuity}}

%%%%%%%%%%%%%%%%%%%%%%%%%%%%%%%%%%%%%%%%%%%%%%%%%%%%%%%

%\begin{proof}
	Fix an oriented edge ${\de}$ of $(G,d)$ and denote by $\ide$ the interior of ${\de}$. We are going to prove that the restriction of $\mu_\mathbb{P}$ to $\ide$ is absolutely continuous with respect to $\mathcal{H}$. To do so, it suffices to prove that this property holds for any sufficiently small open subintervals of $\ide$. Therefore, let us now fix $x\in \ide$ and $I=]a,b[ \ni x$ an open subinterval of $\ide$ such that the graph distance between $a$ and $b$ satisfies $ d(a,b)=|b-a|$; such a $I$ always exists. If $I \neq \ide$ (for instance if $e$ is \emph{not} a minimising edge of $(G,d)$) then turn $G$ into a new graph $G_1$ by adding $a$ and $b$ to the set of vertices and dividing $e$ into three edges in the straightforward way. This new metric graph $G_1$ satisfies the same assumptions as $(G,d)$ (up to reducing the positive lower bound on the edge length). Therefore, without loss of generality, we can assume that $\de$ is a minimising edge. Let us now discuss all the alternatives.

	If $\mu_{\mathbb{P}}$ gives no mass to the set $\ide$, then its restriction on $\ide$
	is null, and thus absolutely continuous. Consider now the case where $\mu_\mathbb{P}(\ide) > 0$ and denote by $\mu_{\ide} \in \mathcal{W}_2(G)$ the normalized probability measure of the restriction of $\mu_\mathbb{P}$ to $\ide$. We still denote by $\Phi$ the branched Wasserstein covering introduced in \Cref{def-theMapPhi}, recall that the properties of $\Phi$ rely on the fact that $\de$ is minimising.

	According to Item 3. of \Cref{thm-prop-branched-cover}, for $\nu \in \was(G)$, $\Phi(\nu)$ is a.c. with respect to the Lebesgue measure iff $\nu \ll \mathcal{H}$. We prove the absolute continuity of $\mu_{\ide}$ by discussing two different cases.

	If $\mu_{\ide} = \mu_\mathbb{P}$,
	then, according to \Cref{lem:barycenter_concentrated_in_the_interior_edge}, $\Phi(\mu_{\ide})$ is the unique barycenter of $\mathbb{Q}: = \Phi_{\#} \mathbb{P}$
	which is absolutely continuous since
	$\mathbb{Q}$ gives mass to absolutely continuous measures on $\mathbb{R}$
	by \Cref{thm-ac-was-R}.
	It follows that $\mu_{\mathbb{P}}$ is absolutely continuous when $\mu_{\ide} = \mu_\mathbb{P}$.
	We now prove that $\mu_{\ide}$ is absolutely continuous when $\mu_{\ide} \neq \mu_\mathbb{P}$.
	For the decomposition $\mu_\mathbb{P} = \lambda \, \mu_{\ide}+ ( 1 - \lambda) \mu^2$
	with $\lambda : = \mu_\mathbb{P}(\ide) \in (0, 1)$
	and $\mu^2 \in \mathcal{W}_2(G)$,
	\Cref{coro:absolute_continuity_and_restrictions_of_barycenters} provides us with
	two measures $\mathbb{P}_1,\mathbb{P}_2 \in \mathcal{W}_2(\mathcal{W}_2(G))$
	such that $\mu_{\ide}$ is a barycenter of $\mathbb{P}_1$
	and $\mu^2$ is a barycenter of $\mathbb{P}_2$.
	Moreover, \Cref{coro:absolute_continuity_and_restrictions_of_barycenters} implies that
	both $\mathbb{P}_1$ and $\mathbb{P}_2$ give mass to absolutely continuous measures
	since $\mathbb{P}$ does so.
	We then have $\mu_{\ide} = \mu_{\mathbb{P}_1}$ is absolutely continuous as proven in the previous case.

	Finally, since our argument holds for arbitrarily chosen ${\de}$, our theorem is proven.
%\end{proof}

%%%%%%%%%%%%%%%%%%%%%%%%%%%%%%%%%%%%%%%%%%%%%%%%%%%

\subsubsection{Proof of \Cref{prop-non-unique}}
We first consider a tripod $T$ made of segments of length 1 glued at the point $v_0$ as in \Cref{fig-trip}.
\begin{figure}[h]
	\centering
	\includegraphics[scale =0.6]{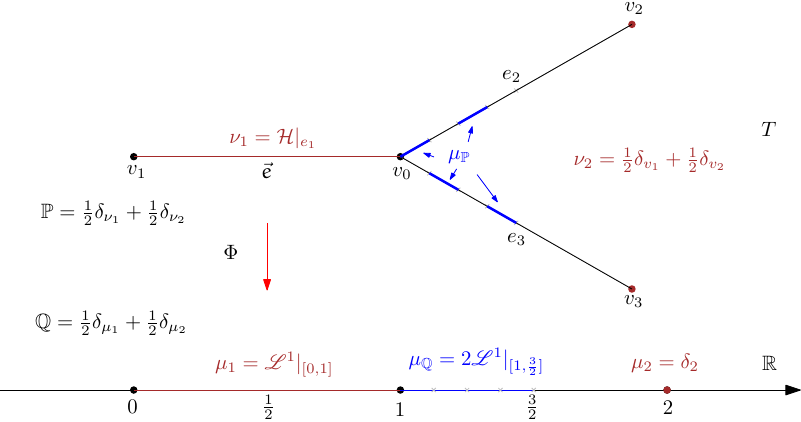}
	\caption{multiple a.c. barycenters}\label{fig-trip}
\end{figure}
Consider the oriented edge ${\de} = \vv{\{ v_1, v_0 \}}$. Define $\nu_1= \mathcal{H}{\res}_{\de}$ and $\nu_2= \frac{1}{2} \delta_{v_2} + \frac{1}{2}\delta_{v_3}$ 
	and let $\Phi : \mathcal{W}_2(T) \rightarrow \mathcal{W}_2(\mathbb{R})$
	be the map associated to $\de$ and $\bar{\mu}= \nu_1$ as in \Cref{def-theMapPhi}. We set $\mathbb{P}= \frac{1}{2} \delta_{\nu_1} + \frac{1}{2} \delta_{\nu_2}.$
	
	Now, consider $\mu_1 : = \Phi(\nu_1) = \Leb^1 |_{[0, 1]}$,
	$\mu_2 : = \Phi(\nu_2) = \delta_2$, and let us introduce 
	
	$$\mu_\mathbb{P} := \frac{1}{2} \eta_2 + \frac{1}{2} \eta_3$$
	with $\eta_2$ (respectively $\eta_3$) being a probability measure
	supported in the edge $e_2$ (respectively $e_3$)
	such that
	\[
		\Phi(\mu_\mathbb{P})
		= \frac{1}{2} \Phi(\eta_2) + \frac{1}{2}\Phi(\eta_3)
		= 2 \Leb^1 |_{[1, \frac{3}{2}]}
	%	= \mu_\mathbb{Q}.
	\]
	
	Our goal is to \emph{prove} that any such $\mu_{\mathbb{P}}$ is a barycenter of $\mathbb{P}$.
	
	\smallskip
	
	We finally define $\mathbb{Q}: = \Phi_{\#}\mathbb{P} = 
	\frac{1}{2} \delta_{\mu_1} + \frac{1}{2} \delta_{\mu_2}$. Recall that when a generic $\mathbb{P}$ gives mass only to two distinct measures with equal weight, any barycenter of $\mathbb{P}$ is a midpoint (relative to $d_W$) of the involved measures and vice versa. Therefore,
	$\mu_\mathbb{Q}: = 2 \Leb^1 |_{[1, \frac{3}{2}]}$ is the barycenter of $\mathbb{Q}$ by its very definition.
	
	Thanks to \Cref{thm-prop-branched-cover}, Item 3., we have
	\begin{equation}\label{eq-toto}
	\begin{array}{rcl}
	 d_W(\nu_1, \mu_{\mathbb{P}})  & = & d_W(\Phi(\nu_1), \Phi(\mu_{\mathbb{P}})) = d_W(\mu_1, \mu_{\mathbb{Q}}) \\   
	 d_W(\nu_1, \nu_2) & = &  d_W(\Phi(\nu_1), \Phi(\nu_2)) = d_W(\mu_1, \mu_{2}). 
	\end{array}
	\end{equation}
	
	The proof is complete if we check that
	$$ d_W(\nu_1, \mu_{\mathbb{P}}) + d_W(\mu_{\mathbb{P}}, \nu_2) = d_W(\nu_1,\nu_2).$$
	According to \eqref{eq-toto}, and since $\mu_{\mathbb{Q}}$ is a midpoint of $\mu_1$ and $\mu_2$, we are done if we prove \\ $d_W(\mu_2, \mu_{\mathbb{Q}})= d_W(\nu_2, \mu_{\mathbb{P}}).$

	To this aim, we use the fact that $e_2 \cup e_3$ is isometric to a segment $[0,2]$ (where $v_2$ is identified with $0$ and $v_3$ with $2$), and the image of $\eta_2$ (resp. $\eta_3$) in $[0,2]$ is a probability measure supported in $[0,1]$ (resp. $[1,2]$).
	Therefore the following  transport plan $\gamma$ between $\mu_{\mathbb{P}}$ and $\nu_2$ is optimal:
	
	\[
		\gamma := \frac{1}{2} \eta_2 \otimes \delta_{v_2} + \frac{1}{2} \eta_3 \otimes \delta_{v_3}.
	\]
	We can now estimate the distance between $\mu_{\mathbb{P}}$ and $\nu_2$:
	\begin{align*}
		d_{W} (\mu_\mathbb{P}, \nu_2)^2
			&= \int_\Gamma d(x, y)^2 \diff \gamma(x, y)\\
		 & = \frac{1}{2} \int_{e_2} d(x, v_2)^2 \diff \eta_2(x)
		+ \frac{1}{2} \int_{e_3} d (x, v_3)^2 \diff \eta_3(x)    \\         &
		          = \frac{1}{2} d_{W}(\Phi(\eta_2), \mu_2)^2
		+ \frac{1}{2} d_{W}(\Phi(\eta_3), \mu_2)^2         \\          &
		           = d_{W}(\Phi(\mu_\mathbb{P}), \mu_2)^2.             \\          &
	\end{align*}
	From our construction of $\mu_\mathbb{P}$,
	there are infinitely many possible choices and all of them are absolutely continuous. We have thus proven that the absolutely continuous part of the Wasserstein barycenter as in \Cref{thm:almost_absolute_continuity} is not unique. 
	
	Let us now prove the non-uniqueness of the atomic part. Let us consider two copies of the tripod described in the introduction and drawn in \Cref{fig:tripod_example}. We obtain a new metric graph $(G,d)$ by gluing each segment $[1/2,1]$ of one copy with the corresponding segment of the other one as in \Cref{fig:glued_tripod}. As we did in the introduction, let $\mathbb{P} := \frac{1}{3}\sum_{i=1}^3 \delta_{\nu_i}$ be a probability measure on the Wasserstein space over the glued tripod, where each $\nu_i$ is an absolutely continuous measure supported on the outer half $[\frac{1}{2}, 1]$ of a distinct branch, for instance the uniform probability measure on this interval. As explained in the introduction in the case of a tripod, the Dirac mass at the common point $0$ is the barycenter of the corresponding $\mathbb{P}$. For each copy of the tripod isometrically embeds into the glued tripod $(G,d)$, we infer that any convex combination of Dirac masses at $0^+$ and $0^-$ is a Wasserstein barycenter of $\mathbb{P}$ and the non-uniqueness is proven.
	
	Note that by isometrically embedding the two previous examples into a single metric graph $(G',d)$ where the images are far from each other, and then by considering a probability measure $\mathbb{P}'$ which is a balanced convex combination of (the images of) the previous $\mathbb{P}$'s, one gets a metric graph where the Wasserstein barycenter has both an atomic and an absolutely continuous part, each of which is non-unique.

\begin{figure}[h]
	\centering
	\includegraphics[ scale=0.6]{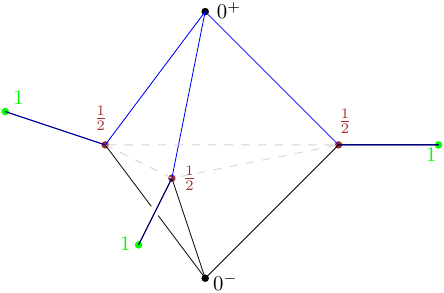}
	\caption{\label{fig:glued_tripod}$\mathbb{P} = \sum_{i=1}^3 \frac{1}{3}\,\delta_{\nu_i}$ on the glued tripod}
\end{figure}

\medskip

\bibliography{bibliography.bib}

\begin{thebibliography}{10}

\bibitem{agueh2011barycenters}
{\sc M.~Agueh and G.~Carlier}, {\em Barycenters in the {W}asserstein space},
  SIAM Journal on Mathematical Analysis, 43 (2011), pp.~904--924.

\bibitem{ahidar2020}
{\sc A.~Ahidar-Coutrix, T.~Le~Gouic, and Q.~Paris}, {\em Convergence rates for
  empirical barycenters in metric spaces: curvature, convexity and extendable
  geodesics}, Probab. Theory Relat. Fields, 177 (2020), pp.~323--368.

\bibitem{ambrosio2021lectures}
{\sc L.~Ambrosio, E.~Bru{\'e}, and D.~Semola}, {\em Lectures on Optimal
  Transport}, vol.~130 of UNITEXT, 2021.

\bibitem{ambrosio2008gradient}
{\sc L.~Ambrosio, N.~Gigli, and G.~Savar{\'e}}, {\em Gradient flows: in metric
  spaces and in the space of probability measures}, Lectures in Mathematics.
  ETH Zürich, Birkhäuser, Basel, 2005.

\bibitem{ambrosio2014calculus}
\leavevmode\vrule height 2pt depth -1.6pt width 23pt, {\em Calculus and heat
  flow in metric measure spaces and applications to spaces with {R}icci bounds
  from below}, Inventiones mathematicae, 195 (2014), pp.~289--391.

\bibitem{benamoub2000}
{\sc J.-D. Benamou and Y.~Brenier}, {\em A computational fluid mechanics
  solution to the {Monge}-{Kantorovich} mass transfer problem}, Numer. Math.,
  84 (2000), pp.~375--393.

\bibitem{bertrand2008existence}
{\sc J.~Bertrand}, {\em Existence and uniqueness of optimal maps on
  {A}lexandrov spaces}, Advances in Mathematics, 219 (2008), pp.~838--851.

\bibitem{bertrand2012geometric}
{\sc J.~Bertrand and B.~R. Kloeckner}, {\em A geometric study of {W}asserstein
  spaces: {H}adamard spaces}, Journal of Topology and Analysis, 4 (2012),
  pp.~515--542.

\bibitem{bigot2018characterization}
{\sc J.~Bigot and T.~Klein}, {\em Characterization of barycenters in the
  {W}asserstein space by averaging optimal transport maps}, ESAIM: Probability
  and Statistics, 22 (2018), pp.~35--57.

\bibitem{bobkovledoux2019}
{\sc S.~Bobkov and M.~Ledoux}, {\em One-dimensional empirical measures, order
  statistics, and {Kantorovich} transport distances}, vol.~1259 of Mem. Am.
  Math. Soc., Providence, RI: American Mathematical Society (AMS), 2019.

\bibitem{bogachev2007measure}
{\sc V.~I. Bogachev}, {\em Measure theory}, Springer, Berlin, 2007.

\bibitem{burago2001course}
{\sc D.~Burago, I.~D. Burago, Y.~Burago, S.~Ivanov, S.~V. Ivanov, and S.~A.
  Ivanov}, {\em A course in metric geometry}, vol.~33, American Mathematical
  Soc., 2001.

\bibitem{cinlar2011probability}
{\sc E.~{\c{C}}{\i}nlar}, {\em Probability and Stochastics}, Graduate Texts in
  Mathematics, Springer New York, 2011.

\bibitem{erbar2022gradient}
{\sc M.~Erbar, D.~L. Forkert, J.~Maas, and D.~Mugnolo}, {\em Gradient flow
  formulation of diffusion equations in the {W}asserstein space over a metric
  graph}, Networks and Heterogeneous Media, 17 (2022).

\bibitem{kuwada2015}
{\sc M.~Erbar, K.~Kuwada, and K.-T. Sturm}, {\em On the equivalence of the
  entropic curvature-dimension condition and {Bochner}'s inequality on metric
  measure spaces}, Invent. Math., 201 (2015), pp.~993--1071.

\bibitem{evans2018measure}
{\sc L.~C. Evans and R.~F. Garzepy}, {\em Measure theory and fine properties of
  functions}, Routledge, Oxfordshire, 2018.

\bibitem{gigli2016optimal}
{\sc N.~Gigli, T.~Rajala, and K.-T. Sturm}, {\em Optimal maps and
  exponentiation on finite-dimensional spaces with {R}icci curvature bounded
  from below}, The Journal of geometric analysis, 26 (2016), pp.~2914--2929.

\bibitem{han2024geometry}
{\sc B.-X. Han, D.~Liu, and Z.~Zhu}, {\em On the geometry of {W}asserstein
  barycenter {I}}, arXiv preprint arXiv:2412.01190,  (2024).

\bibitem{hotz2013sticky}
{\sc T.~Hotz, S.~Skwerer, S.~Huckemann, H.~Le, J.~Marron, J.~C. Mattingly,
  E.~Miller, J.~Nolen, M.~Owen, and V.~Patrangenaru}, {\em Sticky central limit
  theorems on open books},  (2013).

\bibitem{jiang2017absolute}
{\sc Y.~Jiang}, {\em Absolute continuity of {W}asserstein barycenters over
  {A}lesxandrov spaces}, Canadian Journal of Mathematics, 69 (2017),
  pp.~1087--1108.

\bibitem{kim2017wasserstein}
{\sc Y.-H. Kim and B.~Pass}, {\em Wasserstein barycenters over {R}iemannian
  manifolds}, Advances in Mathematics, 307 (2017), pp.~640--683.

\bibitem{le2017existence}
{\sc T.~Le~Gouic and J.-M. Loubes}, {\em Existence and consistency of
  {W}asserstein barycenters}, Probability Theory and Related Fields, 168
  (2017), pp.~901--917.

\bibitem{lisini2007}
{\sc S.~Lisini}, {\em Characterization of absolutely continuous curves in
  {Wasserstein} spaces}, Calc. Var. Partial Differ. Equ., 28 (2007),
  pp.~85--120.

\bibitem{ma-phd}
{\sc J.~Ma}, {\em Régularité des barycentres de Wasserstein}, phd,
  université de Toulouse, 2025.

\bibitem{ma2023absolute}
\leavevmode\vrule height 2pt depth -1.6pt width 23pt, {\em Absolute continuity
  of {Wasserstein} barycenters on manifolds with a lower {Ricci} curvature
  bound}, Calc. Var. Partial Differ. Equ., 65 (2026), p.~36.
\newblock Id/No 9.

\bibitem{pereira2025}
{\sc L.~Manella~Pereira and M.~Hadi~Amini}, {\em A survey on optimal transport
  for machine learning: Theory and applications}, IEEE Access, 13 (2025),
  pp.~26506--26526.

\bibitem{mazon2015optimal}
{\sc J.~M. Maz{\'o}n, J.~D. Rossi, and J.~Toledo}, {\em Optimal mass transport
  on metric graphs}, SIAM Journal on Optimization, 25 (2015), pp.~1609--1632.

\bibitem{mccann2001polar}
{\sc R.~J. McCann}, {\em Polar factorization of maps on {R}iemannian
  manifolds}, Geometric \& Functional Analysis GAFA, 11 (2001), pp.~589--608.

\bibitem{ohta2012barycenters}
{\sc S.-I. Ohta}, {\em Barycenters in {A}lesxandrov spaces of curvature bounded
  below}, Advances in geometry, 12 (2012), pp.~571--587.

\bibitem{otto2001}
{\sc F.~Otto}, {\em The geometry of dissipative evolution equations: {The}
  porous medium equation}, Commun. Partial Differ. Equations, 26 (2001),
  pp.~101--174.

\bibitem{panaretos2020invitation}
{\sc V.~M. Panaretos and Y.~Zemel}, {\em An invitation to statistics in
  {W}asserstein space}, SpringerBriefs in Probability and Mathematical
  Statistics, Springer, Cham, Switzerland, 2020.

\bibitem{paris2020}
{\sc Q.~Paris}, {\em Jensen's inequality in geodesic spaces with lower bounded
  curvature}, Preprintv2,  (2021).

\bibitem{revuz2013continuous}
{\sc D.~Revuz and M.~Yor}, {\em Continuous martingales and Brownian motion},
  vol.~293, Springer Science \& Business Media, 2013.

\bibitem{Santambrogio2015}
{\sc F.~Santambrogio}, {\em Optimal Transport for Applied Mathematicians:
  Calculus of Variations, PDEs, and Modeling}, vol.~87 of Progress in Nonlinear
  Differential Equations and Their Applications, 2015.

\bibitem{sturm2003probability}
{\sc K.-T. Sturm}, {\em Probability measures on metric spaces of nonpositive
  curvature}, SFB 611, 2003.

\bibitem{sturm2005convex}
\leavevmode\vrule height 2pt depth -1.6pt width 23pt, {\em Convex functionals
  of probability measures and nonlinear diffusions on manifolds}, Journal de
  Mathématiques Pures et Appliquées, 84 (2005), pp.~149--168.

\bibitem{villani2021topics}
{\sc C.~Villani}, {\em Topics in optimal transportation}, vol.~58 of Graduate
  Studies in Mathematics, American Mathematical Society, Providence, Rhode
  Island, 2003.

\bibitem{villani2009optimal}
\leavevmode\vrule height 2pt depth -1.6pt width 23pt, {\em Optimal transport:
  old and new}, vol.~338 of Grundlehren der mathematischen Wissenschaften,
  2009.

\end{thebibliography}
\addcontentsline{toc}{section}{Bibliography}

\end{document}